\documentclass{amsart}
\usepackage{enumitem}
\usepackage{amsmath}
\usepackage{amsfonts}
\usepackage{mathtools}
\usepackage{amssymb}
\usepackage{amsthm}
\usepackage[all,cmtip]{xy}
\usepackage{tikz-cd}
\usepackage{adjustbox}
\usepackage{xparse}%
\usepackage{mathtools}

\newtheorem{theorem}{Theorem}
\newtheorem{lemm}[theorem]{Lemma}
\newtheorem{defi}[theorem]{Definition}

\newtheorem{coro}[theorem]{Corollary}
\newtheorem{prop}[theorem]{Proposition}
\newtheorem{theo}[theorem]{Theorem}

\newtheorem{rem}[theorem]{Remark}
\usepackage{fullpage}

\begin{document}

\title{\MakeLowercase{n}-absorbing ideal factorization in commutative rings}
\newcommand{\Ass}{\mbox{\textup{Ass}}} 
\newcommand{\NDC}{\mbox{\textup{NDC}}}
\newcommand{\Spec}{\mbox{\textup{Spec}}}
\newcommand{\Min}{\mbox{\textup{Min}}} 
\newcommand{\height}{\mbox{\textup{ht}}}
\newcommand{\supp}{\mbox{\textup{supp}}}
\newcommand{\reg}{\mbox{\textup{reg}}}
\newcommand{\adeg}{\mbox{\textup{adeg}}}

\author{Hyun Seung choi}
  \address{Department of Mathematics Education, Kyungpook National University, 80 Daehakro,
Bukgu, Daegu 41566, Korea}
\curraddr{}
\email{hchoi21@knu.ac.kr}
\thanks{}


\subjclass[2020]{Primary: 13A15; Secondary: 13B30, 13F05, 13G05}

\keywords{almost pseudo-valuation domains, pseudo-valuation domains, Mori domains, 2-absorbing ideals, strongly Laskerian rings}


\dedicatory{}

\begin{abstract}
In this article, we show that Mori domains, pseudo-valuation domains, and $n$-absorbing ideals, the three seemingly unrelated notions in commutative ring theory,  are interconnected. In particular, we prove that an integral domain $R$ is a Mori locally pseudo-valuation domain if and only if each proper ideal of $R$ is a finite product of 2-absorbing ideals of $R$. Moreover, every ideal of a Mori locally almost pseudo-valuation domain can be written as a finite product of 3-absorbing ideals. To provide concrete examples of such rings, we study rings of the form $A+XB[X]$ where $A$ is a subring of a commutative ring $B$ and $X$ is indeterminate, which is of independent interest, and along with several characterization theorems, we prove that in such a ring, each proper ideal is a finite product of $n$-absorbing ideals for some $n\ge 2$ if and only if $A$ and $B$ are both Artinian reduced rings and the contraction map $\textnormal{Spec}(B)\to\textnormal{Spec}(A)$ is a bijection. A complete description of when an order of a quadratic number field is a locally pseudo valuation domain, a locally almost pseudo valuation domain or a locally conducive domain is given.
\end{abstract}

\maketitle

\section{introduction}
 Pseudo-valuation domains (PVDs) were defined by Hedstrom and Houston \cite{HH}, as a generalization of valuation domains. Due to its several interesting properties, pseudo-valuation domains have been extensively studied \cite{F, HH2}, and their generalizations were also considered. Among them are 
 locally pseudo-valuation domains (LPVDs) considered by Dobbs and Fontana \cite{df83}, and almost pseudo-valuation domains (APVDs) introduced by Badawi and Houston \cite{BH}. 
 
 This paper aims to investigate the ring-theoretic properties of a certain class of commutative rings including that of pseudo-valuation domains, almost pseudo-valuation domains  and locally pseudo-valuation domains
 . In section 2, we gather well-known properties of these classes of integral domains, and prove the APVD-counterparts of known theorems concerning PVDs. Several results proved in this section are used throughout the paper.
 In section 3, we study Mori LPVDs. It is shown that a locally pseudo-valuation domain $R$ is a Mori domain if and only if the complete integral closure $R^{*}$ of $R$ is a Dedekind domain such that the contraction map $\textnormal{Spec}(R^{*})\to \textnormal{Spec}(R)$ is bijective (cf. Corollary \ref{triv58}). We also show that if $R$ is a Mori LPVD, then the Nagata ring of $R$ is a Mori LPVD if and only if the integral closure of $R$ and the complete integral closure of $R$ coincide.
 
 We also relate LPVDs with $n$-absorbing ideals of a commutative ring introduced by Anderson and Badawi in 2011 \cite{Anderson}. Using this concept, the AF-dimension of a ring $R$ can be defined as the smallest $n$ such that each proper ideal of $R$ is a finite product of $n$-absorbing ideals of $R$. Then a Dedekind domain is exactly an integral domain whose AF-dimension is one. We focus on the fact that since every prime ideal is a 2-absorbing ideal, some of the properties of Dedekind domains can be inherited by integral domains with AF-dimension at most two. Motivated by the fact that Dedekind domains possess several interesting properties that connects numerous classes of integral domains, we show that domains with AF-dimension at most two have similar properties.
 In particular, we show that an integral domain $R$ is a Mori LPVD if and only if the AF-dimension of $R$ is at most two. Using this, we show that an LPVD is strongly Laskerian if and only if it is a Mori domain (Lemma \ref{morilas}), extending \cite[Corollary 3.7]{b83}. Motivated by the globalized pseudo-valuation domains (GPVDs) introduced by Dobbs and Fontana \cite{df83}, we provide pullback descriptions of Mori GPVDs.
 
 In section 4, we give a structure theorem of TAF-rings corresponding to \cite[Theorem 39.2]{G}, and show that certain Noetherian domains with finite AF-dimension can be constructed from pullbacks. This result is used to classify orders of quadratic number fields in terms of AF-dimensions and LPVDs, LAPVDs and locally conducive domains, extending \cite[Theorem 2.5]{df87}.
 
 In section 5, we classify reduced rings of the form $A+XB[X]$ whose Krull dimension is 1, where $A\subseteq B$ is an extension of commutative unital rings and $X$ is an indeterminate, in order to provide concrete examples of TAF-rings. In particular, we show that $A+XB[X]$ is a TAF-ring exactly when the contraction map $\textnormal{Spec}(B)\to \textnormal{Spec}(A)$ is a bijection and both $A$ and $B$ are direct product of finitely many fields. 
 
\section{Ring-theoretic properties of almost pseudo-valuation domains}
Throughout this paper, every ring is assumed to be nonzero, commutative and unital. A \textit{quasilocal ring} (respectively, \textit{semi-quasilocal ring}) is a ring with only one  maximal ideal (respectively, finitely many maximal ideals). When $R$ is an integral domain, $K$ will denote its quotient field, and the integral closure (respectively, complete integral closure) of $R$ in $K$ will be denoted by $R'$ (respectively, $R^{*})$. $\mathbb{N}, \mathbb{Z}, \mathbb{Q}$ and $\mathbb{R}$ denote the set of natural numbers, the ring of integers, the field of rational numbers and that of real numbers, respectively. The nilradical (respectively, the set of prime ideals) of a ring $R$ will be denoted by $Nil(R)$ (respectively, $\textnormal{Spec}(R)$). 

\begin{defi}
\label{d1}
Let $R$ be an integral domain with quotient field $K$.
 \begin{enumerate}
     \item An ideal $I$ of $R$ is \textit{strongly prime} (respectively, \textit{strongly primary}) if given $x,y\in K$ with $xy\in I$, either $x\in I$ or $y\in I$ (respectively, either $x\in I$ or $y^{n}\in I$ for some $n\in\mathbb{N}$).
\item   $R$ is said to be a \textit{valuation domain} if the set of ideals of $R$ is totally ordered under set inclusion. 
\item A \textit{DVR} is a Noetherian valuation domain that is not a field.
\item $R$ is a Pr{\"u}fer domain if $R_{M}$ is a valuation domain for each maximal ideal $M$ of $R$.
\item $R$ is a \textit{pseudo-valuation domain}, or \textit{PVD} in short, if every prime ideal of $R$ is strongly prime.
     \item $R$ is an \textit{almost pseudo-valuation domain}, or \textit{APVD} in short, if every prime ideal of $R$ is strongly primary.
     \item An \textit{overring} of $R$ is a ring $T$ such that $R\subseteq T\subseteq K$. An overring $T$ of $R$ is \textit{proper} if $T\neq R$.
     \item By a \textit{valuation overring} of $R$ we mean an overring of $R$ that is also a valuation domain.
     \item Given two $R$-submodules $I, J$ of $K$, the set $\{x\in K\mid xJ\subseteq I\}$ will be denoted by $I:J$.
     \item A nonzero $R$-submodule $I$ of $K$ is said to be a \textit{fractional ideal} if $R:I$ contains a nonzero element.
     \item A fractional ideal $I$ of $R$ is \textit{invertible} if $I(R:I)=R$.
     \item $R$ is \textit{conducive} if each overring of $R$ other than $K$ is a fractional ideal of $R$.
     \item $R$ is \textit{locally conducive} if $R_{M}$ is conducive for each maximal ideal $M$ of $R$.
     \item $R$ is \textit{seminormal} if for any $x\in K$ such that $x^2\in R$ and $x^3\in R$, we have $x\in R$. 
 \end{enumerate}
 \end{defi}
Any unexplained terminology is standard, as in \cite{am69}, \cite{G} or \cite{K}.\\
 \\
 We first collect known results concerning PVDs, APVDs and conducive domains.
 

\begin{theorem}
    Every valuation domain is a PVD, and every PVD is an APVD.
\end{theorem}

\begin{proof}
    The first statement follows from \cite[Proposition 1.1]{HH}, while the second follows from the definition.
\end{proof}

\begin{theorem}
\label{theo1}
    Let $M$ be a maximal ideal of a domain $R$. Then the following are equivalent.
        \begin{enumerate}
        \item $R$ is a PVD (respectively, an APVD).
            \item $M$ is strongly prime (respectively, strongly primary).
            \item $M$ is a prime ideal (respectively, a primary ideal) of a valuation overring of $R$.
            \item $M:M$ is a valuation domain, and $M$ (respectively, the radical ideal of $M$ in $M:M$) is the maximal ideal of $M:M$.
        \end{enumerate}
\end{theorem}

\begin{proof}
Note that if $M:M$ is a valuation domain, then $R$ is quasilocal \cite[Corollary 3.4]{a79}. \\
\\
$(1)\Leftrightarrow(2)$: Follows from \cite[Proposition 3.1 and Corollary 3.6]{a79} and \cite[Theorem 3.4]{BH}.\\ 
\\
$(1)\Leftrightarrow(3)$: Follows from \cite[Proposition 3.11]{a79} and \cite[Theorem 3.4]{BH}.\\
\\
$(1)\Leftrightarrow(4)$: Follows from \cite[Theorem 2.1]{a79} and \cite[Theorem 3.4]{BH}.
\end{proof}

\begin{theo}
\label{vc1}
    An integral domain $R$ is conducive if and only if $R:V\neq (0)$ for some valuation overring $V$ of $R$. In particular, every valuation domain is conducive.
\end{theo}

\begin{proof}
    This is just \cite[Theorem 3.2]{DF}.
\end{proof}

An ideal $I$ of a ring $R$ is \textit{divided} if $I\subsetneq aR$ for each $a\in R\setminus I$. A ring $R$ is \textit{divided} if each prime ideal of $R$ is divided. If $R_{M}$ is a divided ring for each maximal ideal $M$ of $R$, then we say that $R$ is \textit{locally divided}. It is easy to see that Zero-dimensional rings and one-dimensional domains are locally divided. We also have the following result due to \cite[Proposition 2.13.(1) and Theorem 4.1]{BH}.

 \begin{theo}
 \label{theo111} 
 Every APVD is divided and conducive.
 \end{theo}

Given an integral domain $R$ with quotient field $K$, $F(R)$ will denote the set of fractional ideals of $R$. Recall that the map $v:F(R)\to F(R)$ that sends $I$ to $R:(R:I)$ for each $I\in F(R)$, defines the famous $v$-operation on an integral domain $R$ \cite[Chapter 34]{G}. For each $I\in F(R)$, the image of $I$ under $v$ is denoted by $I^{v}$. This is the most famous example of so-called \textit{star operations} that proved to be extremely useful in terms of classification of integral domains, and have become a main topic of multiplicative theory of ideals since Gilmer's modern treatment of star operations \cite[Chapter 32]{G}, which is based upon Krull's work \cite{Kr}. 

If $I^{v}=I$ for a fractional ideal $I$ of $R$, then we say that $I$ is a \textit{divisorial fractional ideal} of $R$. An ideal of $R$ that is also a divisorial fractional ideal of $R$ is said to be a \textit{divisorial ideal} of $R$. A domain $R$ is said to be a \textit{divisorial domain} if every nonzero ideal of $R$ is a divisorial ideal. It is well-known that a valuation domain is a divisorial domain if and only if its maximal ideal is a divisorial ideal. In fact, we have the following.

\begin{theorem}
\label{val1}
\cite[Exercise 34.12]{G}
Let $R$ be a valuation domain with maximal ideal $N$. Then the following are equivalent.
\begin{enumerate}
    \item $N\neq N^2$.
    \item $N$ is a principal ideal of $R$.
    \item $N$ is a divisorial ideal of $R$.
    \item $R$ is a divisorial domain.
\end{enumerate}
Moreover, if $N=N^2$, then $\{aN\mid a\in K\setminus\{0\}\}$ is the set of nondivisorial ideals of $R$.
\end{theorem}

Recall that when $R$ is a PVD, there exists a unique valuation overring $V$ of $R$ such that $\textnormal{Spec}(R)=\textnormal{Spec}(V)$ as sets \cite[Theorem 2.7]{HH}. Such $V$ is called the \textit{associated valuation overring} of $R$. 

\begin{lemm}
\label{l13}
Let $R$ be a domain. Then the following are equivalent.
\begin{enumerate}
\item $R$ is an APVD.
\item $R'$ is a PVD with associated valuation overring $M:M$ for some maximal ideal $M$ of $R$.
\end{enumerate}
In particular, $\textnormal{dim}(R)=\textnormal{dim}(M:M)$ if $R$ is an APVD with maximal ideal $M$.
\end{lemm}

\begin{proof}
(1)$\Rightarrow$(2): Let $R$ be an APVD with maximal ideal $M$. By \cite[Proposition 3.7]{BH}, $R'$ is a PVD, $M$ is an ideal of $R'$, and the radical of $M$ in $R'$ is the maximal ideal of $R'$. Let $N$ (respectively, $N_{0}$) denote the maximal ideal of $R'$ (respectively, $M:M$). Then $N_{0}$ is the radical of $M$ in $M:M$ by Theorem \ref{theo1}, so $M\subseteq N_{0}\cap R'$. It follows that $N\subseteq N_{0}\cap R'$, and we have $N=N_{0}\cap R'$ since $N_{0}\cap R$ is a prime ideal of $R'$. Now $N$ is a prime ideal of $M:M$ \cite[Lemma 1.6]{HH}. It follows that $N$ is the maximal ideal of $M:M$. Since $M:M$ and $N:N$ are valuation overrings of $R'$ with same maximal ideal $N$ by Theorem \ref{theo1}, we must have $N:N=M:M$ \cite[Theorem 17.6]{G}. Hence $M:M$ is the associated valuation overring of $R'$ \cite[Theorem 2.7]{HH}.\\
\\
(2)$\Rightarrow$(1): Suppose that (2) holds. Since $R'$ is a quasilocal domain, so is $R$. Since $M:M$ is a valuation domain, it follows that $R'\subseteq M:M$ \cite[Theorem 10.4]{Mat} and $M$ is an ideal of $R'$. Let $P$ be the radical of $M$ in $R'$. Since $R'$ is a PVD, the set of prime ideals of $R'$ is totally ordered under inclusion \cite[Corollary 1.3]{HH}, and $P$ is a prime ideal of $R'$. Thus $P\cap R=M$, and $P$ must be the maximal ideal of $R'$. Since $M:M$ is the associated valuation overring of $R'$, $P$ is the maximal ideal of $M:M$ by the comment preceding this lemma. Hence, 
the radical of $M$ in $M:M$ is $P$, and $R$ must be an APVD by Theorem \ref{theo1}.\\
\\
The remaining assertion now follows from the fact that $\textnormal{dim}(R)=\textnormal{dim}(R')$.
\end{proof}

\textit{From now on, if $R$ is an APVD, then $M$ (respectively, $N$) will denote the maximal ideal of $R$ (respectively, $R'$), and $V$ the associated valuation overring of $R'$. In this case, we will also call $V$ the associated valuation overring of $R$.} 


\begin{coro}
\label{c13}
Let $R$ be an APVD.
\begin{enumerate}
\item $R$ is a PVD if and only if $M=N$.
\item $R$ is a valuation domain if and only if $R=V$.
\item Either $R=V$, or $M=R:V$ is a divisorial ideal of $R$.
\end{enumerate}
\end{coro}

\begin{proof}
(1): If $M=N$, then $M$ is strongly prime by Lemma \ref{l13}, and $R$ is a PVD. Conversely, if $R$ is a PVD, then $M$ is strongly prime. Since $M$ is a proper ideal of $R'$, it must be a prime ideal of $R'$. Since $M$ is an $N$-primary ideal of $R'$, we must have $M=N$.\\
\\
(2): If $R$ is a valuation domain, then $R$ is a PVD, so $M=N$ by (1). Since $N$ is the maximal ideal of $V$, we have $R=V$ \cite[Theorem 17.6]{G}. The converse is obvious.\\
\\
(3): We have $M\subseteq R:V$ since $M$ is an ideal of both $R$ and $V$. Suppose that $R\neq V$. Then $R:V$ is a proper ideal of $R$, so $R:V\subseteq M$ and $M=R:V$.
On the other hand, $R$ is a divisorial ideal of $R$. Therefore $M=R:V$ is a divisorial ideal of $R$ \cite[Theorem 34.1.(3)]{G}. 
\end{proof}

Given rings $A,B,C$ with unital ring homomorphisms $u:A\to C$ and $v:B\to C$, where $v$ is surjective,  the \textit{pullback} of $u$ with respect to $v$ , denoted by $A\times_{C}B$, is the ring $\{(a,b)\in A\times B\mid u(a)=v(b)\}$. The pullback of our interest in this paper is the case when $B$ is an integral domain, $C=B/I$ for some ideal $I$ of $B$, $A$ is a subring of $C$, and $u$ (respectively, $v$) is the canonical inclusion (respectively, canonical projection). In this case, $D=A\times_{C}B$ is an integral domain, $B$ is an overring of $D$, and $B^{*}=D^{*}$ \cite[Lemma 1.1.4.(10)]{fhp}. Note that such $D$ can be identified to the ring $v^{-1}(A)$ \cite[Lemma 1.1.4.(11)]{fhp}.
The notion of a pullback of a domain was proved to be immensely useful, in terms of producing examples and proving theorems. For instance, it is well-known that every PVD arises from a pullback of valuation domains. Precisely, an integral domain $R$ is a PVD if and only if $R=L\times_{V/M} V$ for some valuation domain $V$ with maximal ideal $M$ and a subfield $L$ of $V/M$ \cite[Proposition 2.6]{ad}. Motivated by this result, we present a pullback characterization of APVDs in the following proposition.

\begin{prop}
\label{pullback1}
Let $R$ be an integral domain. Then $R$ is an APVD if and only if there exists a valuation domain $V$ with maximal ideal $N$, an $N$-primary ideal $M$ and a field $L\subseteq V/M$ such that 
$R=L\times_{V/M}V$ is a pullback
. In this case, $M$ is the maximal ideal of $R$, $V=M:M$ and $L$ is the residue field of $R$.
\end{prop}

\begin{proof}
Suppose that $R$ is an APVD, and let $L=R/M$.  
By the comment preceding this Proposition, we can assume that $R$ is not a PVD. Then $R:V=M$ by Corollary \ref{c13}, so $R$ is the pullback of the given form (cf. \cite[Theorem 1]{bdf86}). 

Conversely, assume that $R$ is a pullback of the given form, and let $\phi':R\to L$ the ring homomorphism induced by the canonical projection map $\phi:V\to V/M$. Then $M=ker(\phi')$ is an ideal of $R$. In fact, $M$ is the maximal ideal of $R$ since $L\cong R/ker(\phi')$. It follows that $V$ is an overring of $R$, and $R$ is an APVD by Theorem \ref{theo1}. Note that $M:M=V$ \cite[Theorem 4.2.6]{fhp}.
\end{proof}

Recall that an integral domain is said to be a \textit{Mori domain} if it satisfies the ascending chain condition on divisorial ideals. Our next goal is to classify Mori APVDs. We first need the following lemma.

\begin{lemm}
\label{step4}
Let $R$ be an APVD such that $N$ is a divisorial fractional ideal of $R$. Then we have the following.
\begin{enumerate}
    \item Each nonzero ideal of $V$ is a divisorial fractional ideal of $R$. 
    \item $J^{v}=JV$ for each nonprincipal ideal $J$ of $R$. 
\end{enumerate}
\end{lemm}

\begin{proof}
If $I$ is an ideal of $V$, then $I^{v}\subseteq N:(N:I)$ \cite[Proposition 1.20]{Pd}. On the other hand, $N:(N:I)=I$ \cite[Proposition 4.1]{BHLP}. Thus $(1)$ follows. For (2), we adapt the proof of \cite[Proposition 2.14]{HH}. Choose a nonprincipal ideal $J$ of $R$. Notice that if $x\in R:J$, then $xJ\subseteq R$, and $xJ\neq R$ since $J$ cannot be a principal ideal of $R$. Hence, $xJ\subseteq M$, and thereby $xJV\subseteq MV=M\subseteq R$ implies $x\in R:JV$. Thus, $R:J\subseteq R:JV$. Since $JV$ is a divisorial fractional ideal of $R$ by (1), we then have $JV=(JV)^{v}\subseteq J^{v}$. Since $J^{v}\subseteq (JV)^{v}=JV$, (2) follows.
\end{proof}


\begin{coro}
\label{moria}
Let $R$ be an integral domain that is not a field. Then the following are equivalent.
\begin{enumerate}
\item $R$ is a Mori APVD.
\item $M:M$ is a DVR for some maximal ideal $M$ of $R$.
\end{enumerate}
\end{coro}
\begin{proof}
(1)$\Rightarrow$(2): Suppose that $R$ is a Mori APVD with maximal ideal $M$. Then $M:M$ is a valuation domain by Proposition \ref{pullback1}. Since a Mori valuation domain is a DVR, we only need to show that $M:M$ is a Mori domain. If $M:M=R$, then we have nothing to prove. Suppose that $M:M\neq R$. 
Then $M$ is a divisorial ideal of $R$ by Corollary \ref{c13}.(3). Hence $M:M$ is a Mori domain \cite[Corollary 11]{b86}.

(2)$\Rightarrow$(1): Suppose that (2) holds, and let $N$ be the maximal ideal of $V=M:M$. Then the radical of $M$ in $V$ is $N$, so $R$ must be an APVD with maximal ideal $M$ by \ref{theo1}. It also follows that $N=aV$ and $M=a^{n}V$ for some $a\in N$ and $n\in\mathbb{N}$. If $R=V$, then $R$ is clearly a Mori domain. Assume that $R\neq V$. Then by Corollary \ref{c13}.(3), $M$ is a divisorial ideal of $R$, and so is $V$. Hence $N$ is a divisorial fractional ideal of $R$. Now,
let $\{I_{\alpha}\}_{\alpha\in \mathcal{A}}$ be a set of ascending chain of divisorial ideals of $R$. By Lemma \ref{step4}, $I_{\alpha}$ is either a principal ideal of $R$ or an ideal of $V$ for each $\alpha\in \mathcal{A}$. Note that $R$ must satisfy the ascending chain condition on principal ideals since every nonunit of $R$ is a nonunit of $V$, while $R$ has the ascending chain property on nonprincipal divisorial ideals since $V$ is a Noetherian ring. Hence, the chain $\{I_{\alpha}\}_{\alpha\in \mathcal{A}}$ must be stationary, and $R$ is Mori.
\end{proof}

\begin{lemm}
\label{mc1}
Let $R$ be an APVD that is not a field. Then the following hold.
\begin{enumerate}
    \item $R^{*}$ is a one-dimensional valuation domain.
    \item $R^{*}$ is the associated valuation overring of $R$ if and only if $R$ is one-dimensional. 
    \item $R$ is Mori if and only if the associated valuation overring of $R$ is a DVR. In this case $R^{*}$ is the associated valuation overring of $R$.
\end{enumerate}
\end{lemm}

\begin{proof}
Let $M$ be the maximal ideal of $R$, $L=R/M$ and $V=M:M$ the associated valuation overring of $R$. Then $R$ is a pullback of $L\times_{V/M}V$ by Proposition \ref{pullback1}. In particular, $R$ and $V$ have the same complete integral closure \cite[Lemma 1.1.4.(10)]{fhp}. Since $R$ is conducive by Theorem \ref{theo111}, $R^{*}$ is completely integrally closed \cite[Corollary 6]{gh66}. Moreover, $R^{*}$ is a valuation domain since it is an overring of a valuation domain $V$. Now (1) and (2) both follow from \cite[Theorem 17.5.(3)]{G}. The first assertion of (3) is an immediate consequence of Corollary \ref{moria}, while the second one then follows from (2) and Lemma \ref{l13}. 
\end{proof}

Following \cite[Definition 2.4]{av96}, a nonunit element $a$ of a ring $R$ is \textit{irreducible} if for each $b,c\in R$ such that $a=bc$, either $aR=bR$ or $aR=cR$.
A ring $R$ is \textit{atomic} if every nonzero nonunit of $R$ can be written in at least one way as a finite product of irreducible elements of $R$. A domain that satisfies the ascending chain condition on principal ideals is atomic, but the converse fails in general \cite{gr74}. Our next result shows that in an APVD, these two properties actually coincide.

\begin{coro}
\label{accp0}
The following are equivalent for an integral domain $R$.
\begin{enumerate}
    \item $R$ is a Mori PVD.
    \item $R$ is a PVD that satisfies the ascending chain condition on principal ideals.
    \item $R$ is an atomic PVD.
    \item There exists a maximal ideal $M$ of $R$ such that $M:M$ is a DVR whose maximal ideal is $M$.
\end{enumerate}
\end{coro}

\begin{proof}
We only need to show (3)$\Rightarrow$(4)$\Rightarrow$(1), which follows from \cite[Theorem 5.1 and Corollary 5.2]{am92} and Corollary \ref{moria}.
\end{proof}


\begin{prop}
\label{accp1}
Let $R$ be an integral domain that is not a PVD. Then the following are equivalent. 
\begin{enumerate}
    \item $R$ is an APVD that satisfies the ascending chain condition on principal ideals.
    \item $R$ is an atomic APVD. 
    \item $M:M$ is a one-dimensional valuation domain for some maximal ideal $M$ of $R$.
\end{enumerate} 
\end{prop}

\begin{proof}
(1)$\Rightarrow$(2) is well-known as mentioned in the comment preceding Corollary \ref{accp0}. Since every APVD is divided by Theorem \ref{theo111},  (2)$\Rightarrow$(3) follows from Lemma \ref{l13} and \cite[Proposition 19]{b99}. Suppose that $R$ satisfies (3). Then $M$ is a nonzero ideal of a one-dimensional valuation domain $M:M$, so  $R$ is an APVD by Theorem \ref{theo1}. By Corollary \ref{c13} and our assumption that $R$ is not a PVD, $M\neq N$. Thus there exists $a\in N\setminus M$, and we have $M\subseteq aV$. Suppose that there exists a strictly ascending chain $(f_{1})\subsetneq (f_{2})\subsetneq\cdots$ of nonzero principal ideals of $R$. Then $f_{i}\in f_{i+1}M\subseteq af_{i+1}V$ for each $i\in\mathbb{N}$. Since $V$ is one-dimensional,  $f_{1}\in \bigcap\limits_{i\in\mathbb{N}}a^{i}V=(0)$ \cite[Theorem 17.1.(3)]{G}, a contradiction. Hence $R$ must satisfy the ascending chain condition on principal ideals. 
\end{proof}

\begin{prop}
\label{omori}
Let $R$ be a Mori PVD with maximal ideal $M$. Then the following are equivalent. \begin{enumerate}
    \item Every overring of $R$ is Mori.
    \item $R'$ is a valuation domain.
    \item Every overring of $R$ is a PVD.
\end{enumerate}
\end{prop}

\begin{proof}
(1)$\Rightarrow$(2): Note that if every overring of $R$ is Mori, then $R'$ is a Dedekind domain \cite[Theorem 3.4]{bd84}. Since every integral overring of $R$ is a PVD with maximal ideal $M$ \cite[Theorem 1.7]{HH}, $R'$ must be a DVR. \\
\\
(2)$\Rightarrow$(1): If $R'$ is a valuation domain, then by Lemmas \ref{l13} and \ref{mc1}, $R'=M:M$ is a DVR with maximal ideal $M$. Let $T$ be an overring of $R$. If $T=R'$ or $T=K$, then it is clearly a Mori domain. Suppose not. Since every overring of $R$ is comparable to $R'$ \cite[Proposition 1.27 and Theorem 1.31]{Gi}, $T$ must be an integral overring of $R$. On the other hand, $M$ is the maximal ideal of both $R$ and $R'$ by Theorem \ref{theo1}, so $M$ is also the maximal ideal of $T$. From Corollary \ref{accp0} it then follows that $T$ is Mori.\\
\\
(2)$\Leftrightarrow$(3): \cite[Corollaire 1.4]{F}.
\end{proof}

\section{TAF-domains, Mori domains and locally pseudo-valuation domains}

In 2007, Badawi \cite{Badawi} generalized the notion of a prime ideal as follows. Let $I$ be an ideal of a ring $R$. We say that $I$ is a \textit{2-absorbing ideal} of $R$ if for any $a_{1},a_{2}, a_{3}\in R$ such that $a_{1}a_{2}a_{3}\in I$, either $a_{1}a_{2}\in I$, $a_{1}a_{3}\in I$ or $a_{2}a_{3}\in I$. It is easy to see that every prime ideal is 2-absorbing, but the converse fails, for the product of two distinct maximal ideals of a ring is 2-absorbing \cite[Theorem 2.6]{Anderson}, but not prime. In \cite{Anderson}, Anderson and Badawi generalized this concept further by defining an ideal $I$ of a ring $R$ to be an \textit{n-absorbing ideal} of $R$ if for any $a_{1},\dots, a_{n+1}\in R$ such that $a_{1}\cdots a_{n+1}\in I$, there exists $i\in \{1,\dots, n+1\}$ such that $\prod\limits_{j\in\{1,\dots, n+1\}\setminus \{i\}}a_{j}\in I$. Given a proper ideal $I$ of a ring $R$, the minimal $n\in\mathbb{N}$ such that $I$ is $n$-absorbing is denoted by $\omega_{R}(I)$, and we set $\omega_{R}(I) = \infty$ when $I$ is not $n$-absorbing for any $n \in \mathbb{N}$. From now on, we will call a ring $R$  \textit{finite-absorbing} if $\omega_{R}(I)\in\mathbb{N}$ for every proper ideal $I$ of $R$.

A ring in which every proper ideal is a finite product of prime ideals is said to be a \textit{general ZPI-ring} \cite[Chapter 39]{G}. On the other hand, Mukhtar et.al. \cite{mad18} considered rings in which each proper ideal is a finite product of 2-absorbing ideals, and called such rings \textit{TAF-rings}. It follows that every general ZPI-ring is a TAF-ring. The main portion of this section consists of various ring-theoretic properties of \textit{TAF-domains}, i.e.,  TAF-rings that are also  integral domains. For instance, it is well-known that an integral domain is a general ZPI-ring (in fact, a Dedekind domain) if and only if it is a Noetherian Pr{\"u}fer domain, and in Proposition \ref{tafmori} we show that a similar criterion holds for TAF-domains. Using this criterion, we prove the structure theorem for TAF-rings analogous to \cite[Theorem 39.2]{G} in section 4.

Recall that an 
integral domain $R$ is a \textit{locally pseudo-valuation domain}, or an $LPVD$ in short, if $R_{M}$ is a PVD for each maximal ideal $M$ of $R$ \cite[Proposition 2.2]{df83}. An integral domain in which each nonzero nonunit is contained in only finitely many maximal ideals is of \textit{finite character}. We begin with a collection of useful facts concerning Mori domains.

\begin{theorem}
\label{theo34}
Let $R$ be an integral domain.
\begin{enumerate}
\item $R$ is a Mori domain if and only if for each nonzero ideal $I$ of $R$, there exists a finitely generated ideal $J$ of $R$ such that $J\subseteq I$ and $R:I=R:J$.
    \item If $R$ is a  one-dimensional Mori domain, then $R$ is of finite character.
    \item If $R$ is Mori, then $R_{S}$ is Mori for each multiplicatively closed subset $S$ of $R$.
    \item Suppose that $R$ is of finite character. Then $R$ is Mori if and only if it is locally Mori.
\end{enumerate}
\end{theorem}

\begin{proof}
(1): This is well-known. For instance, see \cite[Proposition 2.6.11]{e19}.\\
\\
(2): This is exactly \cite[Lemma 3.11]{gr19}.\\
\\
(3): \cite[Proposition 3.3.25]{e19}.\\
\\
(4): By (3), every Mori domain is locally Mori. The converse follows from \cite[Theorem 4.7]{Mat} and \cite[Corollary 5]{z88}.
\end{proof}

\begin{theorem}
\label{bh96t}
Let $R$ be an integral domain. 
\begin{enumerate}
\item If $R$ is a Mori domain and $R:R^{*}\neq (0)$, then $R^{*}$ is a Krull domain. Moreover, $R$ has Krull dimension 1 if and only if $R^{*}$ is a Dedekind domain.
\item If $R$ is a seminormal Mori domain that has Krull dimension 1, then $R^{*}$ is a Dedekind domain.
\item If $R^{*}$ is a Krull domain, then $(R_{S})^{*}=(R^{*})_{S}$ for each multiplicatively closed subset $S$ of $R$.
\end{enumerate}
\end{theorem}

\begin{proof}
\cite[Corollary 18]{b86}, \cite[Corollary 3.4]{bh96}, \cite[Theorem 2.9]{b93} and \cite[Lemma 3.1]{bh96}.
\end{proof}

The following characterization is well-known.

\begin{theo}
\label{dd}
Let $R$ be an integral domain that is not a field. Then the following are equivalent.
\begin{enumerate}
    \item $R$ is a Dedekind domain.
    \item $R$ is of finite character, and $R_{M}$ is a DVR for each maximal ideal $M$ of $R$.
    \item $R$ is a Noetherian Pr{\"u}fer domain.
    \item $R$ is an integrally closed Noetherian locally conducive domain.
\end{enumerate}
\end{theo}

\begin{proof}
(1)$\Leftrightarrow$(2): \cite[Theorem 37.2]{G}.\\
\\
(1)$\Leftrightarrow$(3): \cite[Theorem 37.1]{G}.\\
\\
(3)$\Rightarrow$(4): Let $R$ be a Noetherian Pr{\"u}fer domain. Then $R$ is integrally closed \cite[Theorem 26.2]{G}, and $R_{M}$ is a DVR for each maximal ideal $M$ of $R$ by the equivalence of (2) and (3). Since every valuation domain is conducive by Theorem \ref{vc1}, $R$ is locally conducive.\\
\\
(4)$\Rightarrow$(2): Suppose that (4) holds. 
Let $M$ be a maximal ideal of $R$. Then $R_{M}$ is a conducive Krull domain \cite[Theorems 12.1, 12.4]{Mat}, which is a DVR \cite[Corollary 2.5]{DF}. Consequently, $R$ is a one-dimensional Noetherian domain, and must be of finite character.
\end{proof}

We have an analogous characterization for TAF-domains. 

\begin{prop}
\label{tafmori}
Let $R$ be an integral domain. Then the following are equivalent.
\begin{enumerate}
    \item $R$ is a TAF-domain.
    \item $R$ is of finite character, and $R_{M}$ is a Mori PVD for each maximal ideal $M$ of $R$.
    \item $R$ is a Mori LPVD.
    \item $R$ is a seminormal Mori locally conducive domain.
\end{enumerate}
In particular, $R$ has Krull dimension at most 1.
\end{prop}

\begin{proof}
We may assume that $R$ is not a field.\\
\\
(1)$\Rightarrow$(2): 
Suppose that $R$ is a TAF-domain. Then  $R$ is of finite character and $R_{M}$ is a TAF-domain for each maximal ideal $M$ of $R$ \cite[Theorem 4.4]{mad18}.  Hence $R_{M}$ is an atomic PVD for each maximal ideal $M$ of $R$ by 
\cite[Theorem 4.3]{mad18}, and Corollary \ref{accp0} then yields (2) (cf. \cite[Theorem 4.1]{kimr21}).\\
\\
(2)$\Rightarrow$(3): Follows from Theorem \ref{theo34}.\\
\\
(3)$\Rightarrow$(4): Assume (3). Then $R$ is seminormal \cite[Remarks 2.4]{df83}, and locally conducive \cite[Proposition 2.1]{DF}.\\
\\
(4)$\Rightarrow$(1): Suppose (4) holds, and let $M$ be a maximal ideal of $R$. Then $R_{M}$ is a seminormal Mori conducive domain. Hence $(R_{M})^{*}$ is a conducive Krull domain by Theorem \ref{bh96t}, which is a DVR \cite[Corollary 2.5]{DF}. Hence $R_{M}$ has Krull dimension 1 by Theorem \ref{bh96t}. It follows that $R_{M}$ is a PVD \cite[Corollary 2.6]{DF}, and $R$ is of finite character by Theorem \ref{theo34}. Now $R$ is a TAF-domain by \cite[Theorem 4.4]{mad18}.
\end{proof}

Proposition \ref{tafmori} tells us that every TAF-domain is a Mori domain. In fact, we can say something stronger. Note that an integral domain $R$ is Mori whenever the polynomial ring $R[X]$ is Mori, but the converse fails in general, as Roitman proved \cite[Theorem 8.4]{R90}.
\begin{prop}
\label{polymori}
If $R$ is a TAF-domain, then $R[X]$ is a Mori domain.  
\end{prop}

\begin{proof}
Let $R$ be a TAF-domain. Then $R$ is a Mori LPVD, $R$ has finite character and $R_{M}$ is a Mori domain for each maximal ideal $M$ of $R$ by Proposition \ref{tafmori}, so by \cite[Proposition 3.14]{R89} we may assume that $R$ is quasilocal with maximal ideal $M$. Then $R$ is a Mori PVD, and $V$ is a DVR by Corollary \ref{moria}. Hence $V[X]$ is Mori (in fact, Noetherian). Let $F=R/M$, $L=V/M$, and $T$ be the quotient field of $F[X]$. Then $F[X]=T\cap L[X]$, so $R[X]=R[X]_{M[X]}\cap V[X]$ and $R[X]$ is Mori \cite[Corollary 4.16]{R892}.
\end{proof}



We next ``globalize" Lemma \ref{mc1}.(3). 
A ring extension $R\subseteq T$ is said to be \textit{unibranched} if, for each prime ideal $P$ of $R$, there exists exactly one prime ideal $Q$ of $T$ such that $Q\cap R=P$. 
By a \textit{locally almost pseudo-valuation domain}, or an $LAPVD$ in short, we mean an integral domain $R$ such that $R_{M}$ is an APVD for each maximal ideal $M$.

\begin{coro}
\label{triv58}
Let $R$ be an LAPVD. Then $R$ is a Mori domain if and only if the following two conditions are satisfied. 
\begin{enumerate}
    \item $R^{*}$ is a Dedekind domain.
    \item $R\subseteq R^{*}$ is a unibranched extension.
\end{enumerate}
\end{coro}

\begin{proof}
Assume that $R$ is a Mori LAPVD, and choose a maximal ideal $M$ of $R$. Then $R_{M}$ is a Mori APVD by Theorem \ref{theo34}. We claim that  $(R^{*})_{R\setminus M}$ is a DVR.  To avoid triviality,  $R_{M}$ is assumed to be an integral domain that is not a field. We consider two cases:\\
\\
Case 1: $M$ is an invertible ideal of $R$.\\
\\
If so, then $MR_{M}$ is a principal ideal of $R_{M}$, and $(R_{M})^{*}=MR_{M}:MR_{M}=R_{M}$ is a DVR by Lemma \ref{mc1}. Since $R_{M}\subseteq (R^{*})_{R\setminus M}\subseteq (R_{M})^{*}$, it follows that $(R^{*})_{R\setminus M}$ is a DVR.\\
\\
Case 2: $M$ is not an invertible ideal of $R$.\\
\\
In this case, we have $R:M=M:M$. Moreover, since $R$ is Mori, there exists a finitely generated ideal $J$ of $R$ such that $J\subseteq M$ and $R:M=R:J$ by Theorem \ref{theo34}.(1). Thus $(M:M)_{R\setminus M}\subseteq MR_{M}:MR_{M}\subseteq R_{M}:MR_{M}\subseteq R_{M}:JR_{M}=(R:J)_{R\setminus M}=(R:M)_{R\setminus M}=(M:M)_{R\setminus M}$. Hence $(M:M)_{R\setminus M}=MR_{M}:MR_{M}$ is a DVR   by Corollary \ref{moria}. On the other hand, $R_{M}$ is a one-dimensional quasilocal domain by Lemma \ref{mc1}, so $(R_{M})^{*}\subsetneq K$  \cite[Proposition 4.3.(ii)]{DF}.
Since $(M:M)_{R\setminus M}\subseteq (R^{*})_{R\setminus M}\subseteq (R_{M})^{*}\subsetneq K$, we deduce that $(R^{*})_{R\setminus M}=(R_{M})^{*}$ is a DVR.

Now the claim is proved, and we can see that there exists exactly one prime ideal $N$ of $R^{*}$ such that $N\cap R=M$. Such $N$ is a maximal ideal of $R^{*}$, so $R\subseteq R^{*}$ is a unibranched extension and $(R^{*})_{N'}$ is DVR for each maximal ideal $N'$ of $R^{*}$. Since $R$ is a one-dimensional Mori domain, $R$ is of finite character (Theorem \ref{theo34}.(2)) and so is $R^{*}$. Thus $R^{*}$ is Dedekind by Theorem \ref{dd}.\\
\\
Conversely, suppose that $R$ satisfies $(1)$ and $(2)$, and let $M$ be a maximal ideal of $R$. Then there exists unique maximal ideal $N$ of $R^{*}$ that contracts to $M$, and $M$ has height 1 by $(2)$. Thus $R_{M}$ is a one-dimensional APVD, and by Theorem \ref{bh96t} and Lemma \ref{mc1} $(R_{M})^{*}=(R^{*})_{R\setminus M}=(R^{*})_{N}$ is its associated valuation overring, which is a DVR by (1). Hence $R_{M}$ is a Mori domain. Moreover, since $R^{*}$ is of finite character, so is $R$ by (2). Therefore $R$ is a Mori domain by \cite[Corollary 5]{z88}.
\end{proof}

\begin{rem}
It is well-known that an integral domain is Dedekind if and only if it is an integrally closed Noetherian ring with Krull dimension at most 1 \cite[Theorem 96]{K}. This equivalence cannot be generalized to TAF-domains the way Proposition \ref{tafmori} does. In other words, a TAF-domain is a seminormal Mori domain with Krull dimension at most 1, but the converse fails in general. For instance,
let $R=\mathbb{Q}+(X^2-1)\mathbb{R}[X]$ where $X$ an indeterminate.
Then $R$ is the pullback $\mathbb{Q}\times_{T/M}T$ where $T=\mathbb{R}[X]$ and $M=(X^2-1)T=(X-1)T\cap (X+1)T$. Hence $R$ is a seminormal Mori one-dimensional domain and $M$ is a maximal ideal of $R$ \cite[Proposition 4.1 and Theorem 4.3.(3)]{bh96}. However, $R\subseteq R^{*}=T$ is not a unibranched extension since the prime ideals $(X-1)T$ and $(X+1)T$ of $T$ both lie over $M$. Hence by Proposition \ref{tafmori} and Corollary \ref{triv58}, $R$ cannot be a TAF-domain. 
\end{rem}

Unlike Dedekind domains, a TAF-domain may not be integrally closed. For instance, let $X$ be an indeterminate and $F$ a field that is not algebraically closed. If $L$ is an algebraic closure of $F$,  then $R=F+XL[X]$ is a TAF-domain \cite[Corollary 4.8]{mad18}, but $R$ is not integrally closed. On the other hand, we have the following.

\begin{prop}
An integral overring of a TAF-domain is a TAF-domain.
\end{prop}

\begin{proof}
Let $R$ be a TAF-domain and $T$ an integral overring of $R$. 
If $N$ is a maximal ideal of $T$ and $M=N\cap R$, then $ T_{R\setminus M}$ must be a PVD whose maximal ideal is the maximal ideal of $R_{M}$ \cite[Theorem 1.7]{HH}, because it is an integral overring of $R_{M}$.  Therefore $T_{N}=T_{R\setminus M}$, and $R\subseteq T$ is unibranched. It follows that $T$ is of finite character, since $R$ is of finite character (Proposition \ref{tafmori}). It also follows that the associated valuation overring of $T_{N}$ equals that of $R_{M}$, so $T_{N}$ is Mori by Lemma \ref{mc1}. 
\end{proof}

In the next corollary, we present a relation between various conditions on Mori LPVDs (equivalently, TAF-domains), and extend Proposition \ref{omori}. Recall from \cite{p76} that an integral domain $R$ is an \textit{i-domain} if the contraction map $\textnormal{Spec}(T)\to \textnormal{Spec}(R)$ is injective for each overring $T$ of $R$.
On the other hand, as defined in \cite[p.26]{K},
 an integral domain $R$ is an \textit{S-domain} if $PR[X]$ is a height 1 prime ideal of $R[X]$ whenever $P$ is a height 1 prime ideal of $R$. A ring $R$ is a \textit{strong $S$-ring} if $R/N$ is an $S$-domain for every prime ideal $N$ of $R$. 

\begin{coro}
\label{tafover}
Let $R$ be a Mori LPVD. 
Then the following are equivalent.
\begin{enumerate}
    \item Each overring of $R$ is a Mori LPVD. 
    \item Each overring of $R$ is a Mori domain.
    \item Each overring of $R$ is an LPVD.
    \item $R'=R^{*}$. 
    \item $R'$ is a Dedekind domain. 
    \item $R'$ is a Pr{\"u}fer domain.
    \item $R'$ is a Krull domain.
    \item $R$ is an $i$-domain.
    \item $R$ is a strong $S$-ring. 
    \item $\textnormal{dim}(R[X])\le 2$.
\end{enumerate}
\end{coro}

\begin{proof}
Note first that (1)$\Rightarrow$(2) and (1)$\Rightarrow$(3) are trivial, while (1)$\Rightarrow$(4) and (2)$\Rightarrow$(5) can be derived from Proposition \ref{tafmori} and \cite[Propopsition 3.3, Theorem 3.4]{bd84}. On the other hand, the equivalence of (3), (6) and (8) follows from \cite[Theorem 2.9]{df83}.\\
\\
(4)$\Rightarrow$(5): Follows from Corollary \ref{triv58}.\\
\\
(5)$\Leftrightarrow$(7): Follows from Proposition \ref{tafmori}, \cite[Theorem 3.1]{mad18} and the well-known fact that an integral domain is Dedekind if and only if it is Krull domain with Krull dimension at most 1 \cite[Theorem 12.5]{Mat}.\\
\\
(5)$\Rightarrow$(8): Follows from \cite[Theorem 2.9]{df83}.\\
\\
(8)$\Rightarrow$(1): Suppose that (8) holds, and let $T$ be an overring of $R$. Note first that by Proposition \ref{tafmori} $R$ is an $i$-domain of finite character, and so is $T$. Let $N$ be a maximal ideal of $T$. Then by Proposition \ref{tafmori},
it suffices to show that $T_{N}$ is a Mori PVD. We may assume that $T_{N}\neq K$. Let $M=N\cap R$. Then $M$ is a maximal ideal of $R$ since $R$ is one-dimensional.  Now $R_{M}$ is a Mori PVD $i$-domain. Hence $(R_{M})'$ is the associated valuation overring of $R_{M}$ \cite[Corollary 2.10]{df83}, which is a DVR by Corollary \ref{accp0}
. Thus $T_{N}$ is an integral overring of $R_{M}$ \cite[Proposition 1.16.(3)]{Gi}, and $MR_{M}$ is the maximal ideal of $T_{N}$. It follows that $T_{N}$ is a Mori PVD by Corollary \ref{accp0}. 
\\
\\
(5)$\Rightarrow$(9): If $R'$ is Dedekind, then it is a strong $S$-ring \cite[Proposition 2.5]{m79}, so $R$ is a strong $S$-ring \cite[Corollary 4.7]{m79}.\\
\\
(9)$\Rightarrow$(10): Since $R$ has Krull dimension at most 1  by Proposition \ref{tafmori}, the conclusion follows from \cite[Theorem 39]{K}.\\
\\
(10)$\Rightarrow$(4): Since $R$ has Krull dimension at most 1, $R$ is a strong $S$-ring if and only if $R$ is an $S$-domain. If $\textnormal{dim}(R[X])\le 2$, then for each maximal ideal $M$ of $R$, $\textnormal{dim}(R_{M}[X])=\textnormal{dim}(R[X]_{R\setminus M})\le 2$ and $R_{M}$ is a strong $S$-ring by \cite[Theorem 2.5]{HH2}. Now $(R_{M})'$ is the associated valuation overring of $R_{M}$ \cite[Remark 2.6]{HH2}. Hence $(R')_{R\setminus M}=(R_{M})'=(R_{M})^{*}=(R^{*})_{R\setminus M}$. Since this equality holds for arbitrary maximal ideal $M$ of $R$, we must have $R'=R^{*}$. 
\end{proof}

Recall that an integral domain $R$ is said to be a \textit{globalized pseudo-valuation domain}, or a $GPVD$ in short, if there exists a Pr{\"u}fer overring $T$ of $R$ satisfying the following two conditions.
\begin{enumerate}
    \item $R\subseteq T$ is a unibranched extension.
    \item There exists a nonzero radical ideal $A$ common to $T$ and $R$ such that
each prime ideal of $T$ (respectively, $R$) which contains $A$ is a maximal ideal
of $T$ (respectively, $R$).
\end{enumerate}

As mentioned in \cite[p.155-156]{df83}, the class of GPVDs is a stronger globalization of that of PVDs than that of LPVDs, in the sense that every GPVD is an LPVD, and given a maximal ideal $M$ of $R$, there exists unique maximal ideal $N$ of $T$ such that $T_{N}$ is the associated valuation overring of $R_{M}$. In this case, $T$ is uniquely determined by the above conditions, and is called the \textit{Pr{\"u}fer domain associated to $R$}.  From now on, when $R$ is a GPVD, $T$ will denote the Pr{\"u}fer domain associated to $R$.
However, note that even a Noetherian LPVD may not be a GPVD \cite[Example 3.4]{df83}.


\begin{prop}
Let $R$ be a GPVD. Then $R$ is Mori if and only if $T$ is a Dedekind domain. 
\end{prop}

\begin{proof}
Suppose that $R$ is Mori. Since every GPVD is an LPVD, the conclusion follows from Corollary \ref{tafover}. Conversely, suppose that $T$ is Dedekind. Then given a maximal ideal $M$ of $R$, there exists unique maximal ideal $N$ of $T$ such that $T_{N}$ is the associated valuation overring of $R_{M}$. Since $(R_{M})^{*}=(T_{N})^{*}$,  
\end{proof}

\begin{lemm}
\label{66}
Let $R$ be a Mori domain. Then $R$ is a GPVD if and only if $R:R^{*}\neq (0)$ and $R$ is an LPVD. In this case, $T=R^{*}$ and $A=R:R^{*}$, where $T$ and $A$ are as mentioned in the definition of a GPVD. 
\end{lemm}

\begin{proof}
Suppose that $R$ is a GPVD. Given a maximal ideal $M$ of $R$, $R_{M}$ is a Mori PVD, and there exists the unique prime ideal $N$ of $T$ such that $N\cap R=M$. Notice that $(R_{M})^{*}$ is the associated valuation overring of $R_{M}$ by Lemma \ref{mc1}. Then by \cite[Lemma 3.1]{b93} and Corollary \ref{triv58} we have $(R^{*})_{R\setminus M}=(R_{M})^{*}=T_{N}=T_{R\setminus M}$, and $T=R^{*}$ by globalization. $R:R^{*}\neq (0)$ then follows from the definition of GPVDs.\\
\\
Conversely, suppose that $R:R^{*}\neq (0)$ and $R$ is an LPVD, and let $I=R:R^{*}$.  Then $R$ is a TAF-ring by Proposition \ref{tafmori}, so $R^{*}$ is a Dedekind domain and $R\subseteq R^{*}$ is a unibranched extension by Corollary \ref{triv58}. If $I=R$, then $R=R^{*}$ is a GPVD. Suppose that $I$ is a proper ideal of $R$. Since $R^{*}$ is a fractional ideal of $R$, $R_{S}:(R^{*})_{S}=I_{S}$ for each multiplicatively closed subset $S$ of $R$ by Theorem \ref{theo34}.(1). On the other hand, since $R^{*}$ is a Dedekind domain, there exist $a_{1},\dots, a_{n}\in\mathbb{N}$ and maximal ideals $N_{1},\dots, N_{n}$ of $R^{*}$ such that $I=N_{1}^{a_{1}}\cdots N_{n}^{a_{n}}$. Let $M_{i}=N_{i}\cap R$ for each $i\in\{1,\dots, n\}$. Then $\{M_{1},\dots, M_{n}\}$ is the set of maximal ideals of $R$ containing $I$, and $IR_{M_{i}}=R_{M_{i}}:(R^{*})_{R\setminus M_{i}}=R_{M_{i}}:(R_{M_{i}})^{*}=M_{i}R_{M_{i}}$ by Theorem \ref{bh96t}.(3) and Lemma \ref{mc1}.(3). Thus $a_{i}=1$ for each $i$, and $I$ is an intersection of maximal ideals of $R^{*}$. Thus $I$ is a common radical ideal of $R$ and $R^{*}$. Moreover, $I=M_{1}\cap\cdots\cap M_{n}$, so each prime ideal of $R^{*}$ (respectively, $R$) which contains $I$ is a maximal ideal of $R^{*}$ (respectively, $R$). 
We conclude that $R$ is a GPVD with $R^{*}$ its associated Pr{\"u}fer domain.
\end{proof}

Recall that a ring in which each proper ideal is a finite intersection of primary ideals is said to be \textit{Laskerian}. A Laskerian ring in which each primary ideal contains a power of its radical is said to be \textit{strongly Laskerian}. 
In the next lemma, we record that an LPVD is Mori if and only if it is strongly Laskerian, extending a result of Barucci \cite[Corollary 3.7]{b83}.

\begin{lemm}
\label{morilas}
The following are equivalent when $R$ is an LPVD.
\begin{enumerate}
\item $R$ is strongly Laskerian.
    \item $R$ is finite-absorbing.
    \item $R$ is Mori.
\end{enumerate}
\end{lemm}

\begin{proof}
(1)$\Rightarrow$(2): This is \cite[Lemma 19]{c211}.\\
\\
(2)$\Rightarrow$(3): Note that by \cite[Lemma 30]{c211} and Proposition \ref{tafmori}, we may assume that $R$ is a PVD with maximal ideal $M$. But then the conclusion follows from \cite[Corollary 54]{c211} and Corollary \ref{accp0}.\\
\\
(3)$\Rightarrow$(2): Follows from Proposition \ref{tafmori}.
\end{proof}


As we have seen from Proposition \ref{tafmori}, TAF-domains have some interesting ring-theoretic properties. On the other hand, in \cite{mad18}, the authors focus on Noetherian TAF-domains. In particular, they characterized when a Noetherian domain $R$ with $R:R'\neq (0)$ is a TAF-domain \cite[Corollary 4.10]{mad18}. Note that if $R$ is a Noetherian domain, then $R$ is Mori and $R'=R^{*}$. The next theorem is motivated by this observation. Recall that a ring is \textit{reduced} if its zero ideal is a radical ideal.


\begin{theorem}
\label{tafpullback}
(cf. \cite[Corollary 4.10]{mad18}) 
Let $R$ be an integral domain that is not a field. Then the following are equivalent.
\begin{enumerate}
\item $R$ is a TAF-domain  such that $R:R^{*}\neq (0)$.
\item $R$ is a Mori GPVD.
\item There exists a Dedekind domain $T$, distinct maximal ideals $N_{1},\dots, N_{n}$ of $T$ and field extensions $K_{i}\subseteq T/N_{i}$  such that $R$ is a pullback domain $\pi^{-1}(\prod\limits_{i=1}^{n}K_{i})$ where $\pi: T\to \prod\limits_{i=1}^{n} (T/N_{i})$ is the canonical map.
\item $R^{*}$ is a Dedekind domain and $R/(R:R^{*})\subseteq R^{*}/(R:R^{*})$ is a unibranched extension of Artinian reduced rings.
\end{enumerate}
\end{theorem}

\begin{proof}
(1)$\Leftrightarrow$(2): Follows from Proposition \ref{tafmori} and Lemma \ref{66}.\\
\\
(2)$\Rightarrow$(3): If (2) holds, then $R$ is a Mori GPVD with $R^{*}$ its  associated Pr{\"u}fer domain by Lemma \ref{66}. Let $T=R^{*}$ and let $A$ be the ideal mentioned in the definition of a GPVD. Since $T=R^{*}$ is a Dedekind domain by Corollary \ref{triv58}, $A$ is a finite product of maximal ideals of $T$, say, $N_{1},\dots, N_{n}$. Note that since $A$ is a radical ideal of a Dedekind domain, $N_{1},\dots, N_{n}$ are distinct maximal ideals of $T$. If we let $M_{i}=N_{i}\cap R$ and $K_{i}=R/M_{i}$ for each $i\in\{1,\dots, n\}$, then $A=\bigcap\limits_{i=1}^{n}N_{i}=\bigcap\limits_{i=1}^{n}M_{i}$, and $R$ is the pullback	of the form stated in (2) (cf. \cite[Theorem 3.1]{df83}). \\
\\
(3)$\Rightarrow$(2) and (3)$\Rightarrow$(4): Suppose that (3) holds. Then $T$ is an overring of $R$, and $R^{*}=T^{*}=T$ by \cite[Lemma 1.1.4.(10)]{fhp} 
 and the fact that Dedekind domains are completely integrally closed. By Chinese remainder theorem it follows that $R:T=N_{1}\cap\cdots\cap N_{n}$. Let $M$ be a maximal ideal of $R$. If $R:T\subseteq M$, then $M=N_{i}\cap R$ for some $i\in\{1,\dots, n\}$ \cite[Lemma 1.1.4.(6)]{fhp}, and $R_{M}$ is the pullback $\pi_{i}^{-1}(K_{i})$ where $\pi_{i}: T_{N_{i}}\to T_{N_{i}}/N_{i}T_{N_{i}}$ is the canonical map (cf. \cite[Lemma 1.1.6]{fhp}). Therefore $R_{M}$ is a PVD by \cite[Proposition 2.6]{ad}. On the other hand, if $R:T\not\subseteq M$, then there exists unique prime ideal $N$ of $T$ such that $M=N\cap R$, and $R_{M}$ is isomorphic to $T_{N}$\cite[Lemma 1.1.4.(3)]{fhp}, which is a DVR. It also follows that $R$ is an LPVD and $R\subseteq T$ is a unibranched extension. Hence by Corollary \ref{triv58}, $R$ is Mori. Now Lemma \ref{66}, $R$ is a GPVD and (2) follows. Since $R\subseteq T$ is unibranched, so is $R/R:T\to T/R:T$. Since $R:T$ is an intersection of $n$ maximal ideals (as an ideal of both $R$ and $T$), $R/R:T$ and $T/R:T$ are both isomorphic to a product of $n$ fields by Chinese remainder theorem, so they are Artinian reduced rings, and (4) follows. \\
 \\
 (4)$\Rightarrow$(3): Assume (4). Since $R/(R:R^{*})$ is Artinian, if $R:R^{*}$ is the zero ideal, then $R$ is a field, which is a contradiction. Therefore $R:R^{*}\neq (0)$. Since $R^{*}$ is a Dedekind domain, we have $R:R^{*}=N_{1}\cdots N_{n}$ for some distinct maximal ideals of $R^{*}$. Letting $T=R^{*}$, we have (3).
\end{proof}


Let $R$ be an integral domain. Given $f\in R[X]$, let $c(f)$ be the ideal of $R$ generated by the coefficients of $f$. Then $N=\{f\in R[X]\mid c(f)=R\}$ and $N_{v}=\{f\in R[X]\mid (c(f))^{v}=R\}$ are  multiplicatively closed subsets of $R[X]$ \cite[Proposition 2.1]{k89}. $R[X]_{N}$ is usually denoted by $R(X)$, and is called the \textit{Nagata ring of }$R$. 

\begin{prop}
\label{nagata2}
The following are equivalent for an integral domain $R$.
\begin{enumerate}
    \item $R(X)$ is a TAF-domain.
     \item Every overring of $R(X)$ is a TAF-domain.
    \item Every overring of $R$ is a TAF-domain.
\end{enumerate}
\end{prop}

\begin{proof}
(1)$\Rightarrow$(3): Assume that $R(X)$ is a TAF-domain. Then $R$ is an LPVD, $R'$ is a Pr{\"u}fer domain and every overring of $R(X)$ is an LPVD \cite[Corollary 3.9]{c08}. Then $R[X]_{N_{v}}$, being an overring of $R(X)$, is a Mori LPVD by Corollary \ref{tafover}. Since a strictly ascending chain of divisorial ideals $\{I_{i}\}_{i\in\mathbb{N}}$ of $R$ induces a strictly ascending chain of divisorial ideals $\{I_{i}[X]_{N_{v}}\}_{i\in\mathbb{N}}$ of $R[X]_{N_{v}}$ by \cite[Proposition 2.8]{k89}, $R$ must be a Mori domain. It follows that each overring of $R$ is a TAF-domain by Corollary \ref{tafover}.\\
\\
(3)$\Rightarrow$(2): Suppose that every overring of $R$ is a TAF-domain. Then by Proposition \ref{polymori}, $R[X]$ is a Mori domain, and so is $R(X)$, being a localization of $R[X]$ (Theorem \ref{theo34}.(3)). On the other hand, $R$ is an LPVD and $R'$ is a Dedekind domain by Corollary \ref{tafover}, so each overring of $R(X)$ is an LPVD by \cite[Corollary 3.9]{c08}. Therefore by Proposition \ref{tafmori} and Corollary \ref{tafover}, each overring of $R(X)$ is a TAF-domain.\\
\\
(2)$\Rightarrow$(1): Trivial.
\end{proof}







\section{TAF-rings and FAF-domains}
The main theorem of this section is Theorem \ref{taf54} which generalizes Proposition \ref{tafmori} to commutative rings with zero divisors. The key part of its proof is taken from \cite[Theorem 5.1]{am92}. 
 A prime ideal $P$ of a commutative ring $R$ is said to be \textit{strongly prime} if $aP$ and $bR$ are comparable for any two elements $a,b$ of $R$. Note that for an integral domain $R$, this definition coincides with the notion of strongly prime ideal in Definition \ref{d1} (cf. \cite[Proposition 3.1]{a79}). Similarly, a ring $R$ is said to be a \textit{pseudo-valuation ring} or a $PVR$ if some maximal ideal of $R$ is strongly prime \cite[Lemma 1 and Theorem 2]{abd97}. A ring $R$ is said to be a \textit{locally pseudo-valuation ring} or an $LPVR$ if $R_{M}$ is a PVR for each maximal ideal $M$ of $R$. 

\begin{lemm}
\label{l54}
Let $R$ be a quasilocal ring with maximal ideal $M$. 
\begin{enumerate}
    \item For any $r,s\in R$ such that $r=rs$, either $r=0$ or $s$ is a unit of $R$.
    \item Let $a$ be a nonzero irreducible element of $R$. If $a=bc$ for some $b,c\in R$, then one of $b$ and $c$ is a unit of $R$.
    \item $R$ is a PVR if and only if for any two ideals $I,J$ of $R$, $I$ and $JM$ are comparable.
    \item Let $R$ be a PVR. Then every ideal of $R$ is comparable to $M^2$, and $M^2=aM$ for each nonzero irreducible element $a$ of $R$.
\end{enumerate}
\end{lemm}

\begin{proof}
(1): If $r=rs$, then $r(1-s)=0$. If $s$ is not a unit of $R$, then $1-s$ is a unit of $R$ since $R$ is quasilocal. Therefore $r=0$.\\
\\
(2): If $a=bc$, then without loss of generality we have $aR=bR$, and $b=ar$ for some $r\in R$. Now $a=arc$, so $c$ is a unit of $R$ by (1). \\
\\
(3) This follows from \cite[Theorem 5]{abd97}.\\
\\
(4) 
The first assertion follows from (3). For the second assertion, let $a$ be a nonzero irreducible element of $R$. Then for each $b\in M$, $a\not\in bM$ by (2). Since $R$ is a PVR, we must have $bM\subseteq aR$, from which it follows that $M^2\subseteq aR$. Then $M^2=aI$ for some ideal $I$ of $R$. Since $M^2\neq aR$, $I\subseteq M$ and $aI\subseteq aM\subseteq M^2$. Therefore $M^2=aM$.
\end{proof}

In the next lemma, we extend Corollary \ref{accp0} and part of \cite[Theorem 4.3]{mad18} to commutative rings with zero divisors.

\begin{lemm}
\label{l56}
Let $R$ be a quasilocal ring with maximal ideal $M$. Then the following are equivalent.
\begin{enumerate}
    \item $R$ is a TAF-ring.
    \item $R$ is strongly Laskerian and every ideal of $R$ is comparable to $M^2$.
    \item $R$ is a strongly Laskerian PVR.
    \item $R$ is a PVR that satisfies the ascending chain condition on principal ideals.
    \item $R$ is an atomic PVR.
    \item $R$ is an atomic ring, and for each nonzero proper ideal $I$ of $R$ there exists $n\in\mathbb{N}$ such that $M^{n}\subseteq I\subsetneq M^{n-1}$.
\end{enumerate}
\end{lemm}

\begin{proof}
Note that strongly Laskerian rings satisfy the ascending chain condition on principal ideals \cite[Corollary 3.6.(b)]{hl83}, and every ring that satisfies the ascending chain condition on principal ideals is atomic  \cite[Theorem 3.2]{av96}. Therefore we have $(3)\Rightarrow (4)\Rightarrow (5)$.\\
\\
$(5)\Rightarrow (6)$: Let $R$ be an atomic PVR and $I$ a nonzero proper ideal of $R$. Choose a nonzero $a\in I$. Since $R$ is atomic, $a=a_{1}\cdots a_{n}$ for some (nonzero) irreducible elements $a_{1},\dots, a_{m}$ of $R$. Then $M^{m+1}=aM\subseteq I$, where the first equality follows from Lemma \ref{l54}.(3). Thus there exists the smallest $n\in\mathbb{N}$ such that $M^{n}\subseteq I$. If $n=1$, then we are done. If $n\ge 2$, then $M^{n-2}M=M^{n-1}\not\subseteq I$, so $I\subsetneq M^{n-2}M=M^{n-1}$ by Lemma \ref{l54}.(3).\\
\\
$(6)\Rightarrow (5)$: Assume $(6)$, and let $a,b\in R$. Then we only need to show that $aM$ and $bR$ are comparable. We may assume that $a,b$ are nonzero nonunits of $R$. Then $M^{n}\subseteq aR\subsetneq M^{n-1}$ and $M^{m}\subseteq bR\subsetneq M^{m-1}$ for some $n,m\in\mathbb{N}$. If $n<m$, then $bR\subsetneq M^{m-1}\subseteq M^{n}\subseteq aR$. It follows that $bR=aI$ for some proper ideal $I$ of $R$, so $bR\subseteq aM$. On the other hand, if $n\ge m$, then $aM\subseteq M^{n}\subseteq M^{m}\subseteq bR$.\\ 
\\
$(5)\Rightarrow(1)$: Suppose that $(5)$ holds. If $0$ is irreducible in $R$, then $R$ is an integral domain and the conclusion follows from Corollary \ref{accp0} and Proposition \ref{tafmori}. Assume that $0$ is not irreducible in $R$. Since $R$ is atomic, $0=a_{1}\cdots a_{n}$ for some (nonzero) irreducible elements $a_{1},\dots, a_{n}$ of $R$. Then $M^{n+1}=a_{1}\cdots a_{n}M=(0)$, where the first equality follows from Lemma \ref{l54}.(3). Thus $M$ is nilpotent, and $R$ is finite-absorbing \cite[Theorem 27]{c211}. Since every ideal of $R$ is comparable to $M^2$ by Lemma \ref{l54}.(3), $R$ is a TAF-ring \cite[Proposition 33]{c211}.\\
\\
$(1)\Rightarrow (2)$: \cite[Proposition 33]{c211}.\\
\\
$(2)\Rightarrow(3)$. Suppose that (2) holds. If $M=M^2$, then $M=\{x\in R\mid x\in xM\}$ \cite[Exercise 29.(d), Chapter IV, \S 2]{Bourbaki}, so $M=(0)$ and $R$ is a field. Hence we may assume that $M\neq M^2$. Notice that by Lemma \ref{l54}.(3), given irreducible elements $a_{1},\dots, a_{n}$ of $R$, $a_{1}\cdots a_{n}M=M^{n+1}$. Now choose $a,b\in R$. We need to show that $aM$ and $bR$ are comparable. We may assume that $a, b\in M\setminus\{0\}$. Note that $R$ is atomic as mentioned in the beginning of this proof, so there exist irreducible elements $a_{1},\dots, a_{n}, b_{1},\dots, b_{m}$ of $R$ such that $a=a_{1}\cdots a_{n}$ and $b=b_{1}\cdots b_{m}$. Suppose that $n<m$. Then $bR=b_{1}\cdots b_{m}R\subseteq M^{m}\subseteq M^{n+1}=a_{1}\cdots a_{n}M=aM$. On the other hand, if $n\ge m$, then $aM=a_{1}\cdots a_{n}M=M^{n+1}\subseteq M^{m+1}=b_{1}\cdots b_{m}M\subseteq b_{1}\cdots b_{m}R=bR$. Hence $M$ is strongly prime, and $R$ is a PVR \cite[Theorem 2]{abd97}.
\end{proof}

Now we can derive the promised result.

\begin{theorem}
(cf. \cite[Theorem 39.2]{G})
\label{taf54}Let $R$ be a ring. Then the following are equivalent.
\begin{enumerate}
    \item $R$ is a TAF-ring.
        \item $R$ is strongly Laskerian, and for each maximal ideal $M$ of $R$, every $M$-primary ideal of $R$ is comparable to $M^2$.
            \item $R$ is a strongly Laskerian LPVR.
    \item $R$ is a finite-absorbing LPVR.
    \item $R=R_{1}\times\cdots\times R_{r}$, where $R_{i}$ is either an atomic PVR or a Mori LPVD for each $i\in\{1,\dots, r\}$.
\end{enumerate}
\end{theorem}

\begin{proof}
We may assume that $R$ is not a field. \\
\\
(1)$\Rightarrow$(2): Follows from \cite[Proposition 33]{c211}.\\
\\
(2)$\Rightarrow$(3): We mimic the proof  of \cite[Theorem 5.1]{am92}. Note first that if a ring $R$ satisfies $(2)$, then so is $R_{M}$ for each maximal ideal $M$ of $R$. 
Hence we may assume that $R$ is a quasilocal ring that satisfies (2), and the conclusion follows from Lemma \ref{l56}.\\
\\
(3)$\Leftrightarrow$(4): Note that every LPVR is locally divided \cite[Lemma 1.(a)]{abd97}. Hence (3)$\Leftrightarrow$(4) follows from \cite[Lemma 20.(2)] {c211}.\\
\\
(4)$\Rightarrow$(1): By \cite[Theorem 2.5]{Anderson} and \cite[Corollary 32]{c211}, we may assume that $R$ is quasilocal, so Lemma \ref{l56} yields the conclusion.\\
\\
(1)$
\Leftrightarrow$(5): Follows from \cite[Proposition 2.4, Theorem 3.3]{mad18}, Proposition \ref{tafmori} and Lemma \ref{l56}.
\end{proof}



In \cite{adk}, the authors introduced the \textit{AF-dimension} of a ring $R$, denoted by AF-dim$(R)$, which is the minimum positive integer $n$ such that every proper ideal of $R$ can be written as a finite product of $n$-absorbing ideals of $R$ (if such $n$ does not exist, set AF-dim$(R)=\infty$). We call $R$ an \textit{FAF-ring} (finite absorbing factorization ring) if AF-dim$(R)$ is finite. An integral domain that is also an FAF-ring will be called an \textit{FAF-domain}. The authors of \cite{adk} themselves presented several examples of rings and computed their AF-dimensions. All of the rings considered in such examples, however, were Noetherian. In the remainder of this section, motivated by the result that an integral domain $R$ is a Mori LPVD if and only if $\textnormal{AF-dim}(R)\le 2$ (Proposition \ref{tafmori}), we show that  $\textnormal{AF-dim}(R)\le 3$ whenever $R$ is a Mori LAPVD, and construct a non-Noetherian  example that attains the equality. We also prove that an integral domain $R$ may have AF-dimension 3 without being Mori.

The following result enables us to compute the AF-dimension of a domain locally.

\begin{lemm}
\label{c31c}
(cf. \cite[Corollary 31]{c211}, \cite[Theorem 4.3]{adk})
Let $R$ be an integral domain and $n\in\mathbb{N}$. Then the following are equivalent.
\begin{enumerate}
    \item $\textnormal{AF-dim}(R)\le n$.
    \item $R$ is of finite character and $\textnormal{AF-dim}(R_{M})\le n$ for each maximal ideal $M$ of $R$.
\end{enumerate}
\end{lemm}

\begin{theorem}
\label{faf1}
(cf. \cite[Theorem 5.4]{adk})
Let $R$ be a Mori domain such that $R:R^{*}$ is nonzero. Then the following are equivalent.
\begin{enumerate}
\item $R$ is an FAF-domain.
\item $R_{M}$ is an FAF-domain for each maximal ideal $M$ of $R$.
\item  $R\subseteq R^{*}$ is a unibranched extension of one-dimensional domains.
\item $R$ is locally conducive.
\end{enumerate}
\end{theorem}

\begin{proof} 
We may again assume that $R$ is not a field. Note that if $R$ satisfies one of $(1),(2), (3)$ and $(4)$, then $R$ has Krull dimension 1 by \cite[Theorem 4.1]{adk}, \cite[Theorem 2.2]{bd84} and Theorem \ref{theo34}.(3), so $R^{*}$ is Dedekind and $(R^{*})_{R\setminus M}=(R_{M})^{*}$ for each maximal ideal $M$ of $R$ by Theorem \ref{bh96t}.\\
\\
(1)$\Rightarrow$(2): Follows from \cite[Proposition 3.5]{adk}.\\
\\
(2)$\Rightarrow$(3): Suppose that $R$ satisfies (2), and fix a maximal ideal $M$ of $R$. We have to show that the contraction map $f: \textnormal{Spec}(R^{*})\to \textnormal{Spec}(R)$ is bijective. Since $R$ has Krull dimension 1, the surjectivity of $f$ follows from \cite[Proposition 1.1]{bh96}. On the other hand, by Theorem \ref{theo34}.(3) 
 $R_{M}$ is a Mori FAF-domain such that $R_{M}:(R_{M})^{*}\supseteq (R:R^{*})_{ R\setminus M}\neq (0)$. Hence by \cite[Lemma 5.2]{adk}, there exists only one maximal ideal of $(R^{*})_{R\setminus M}$ that contracts to $MR_{M}$. Thus $f$ is injective.\\
\\  
(3)$\Rightarrow$(4) and $(1)$: Suppose that  $R\subseteq R^{*}$ is a unibranched extension of one-dimensional domains. Since $R^{*}$ is of finite character, so is $R$. Thus there are only finitely many maximal ideals of $R$ that contains $R:R^{*}$. Consider a maximal ideal $M$ of $R$, and let $N$ be the maximal ideal of $R^{*}$ that contracts to $M$. If $M$ does not contain $R:R^{*}$, then $(R:R^{*})R_{M}=R_{M}$, and $(R^{*})_{R\setminus M}=(R^{*})_{R\setminus M}R_{M}=(R^{*})_{R\setminus M}(R:R^{*})R_{M}\subseteq R_{M}$. Hence $R_{M}=(R^{*})_{R\setminus M}=(R^{*})_{N}$ is a DVR, which is conducive by Theorem \ref{vc1}. If $R:R^{*}\subseteq M$, then $R_{M}:(R^{*})_{N}\supseteq (R:R^{*})_{M}\neq (0)$ and $(R^{*})_{N}=(R_{M})^{*}$ is a DVR that is also an overring of $R_{M}$. Therefore $R_{M}$ is a conducive domain by Theorem \ref{vc1}, and $(4)$ follows. Moreover, $R_{M}$ is an FAF-domain since $N(R^{*})_{N}\cap R_{M}=MR_{M}$\cite[Lemma 5.1]{adk}. By Lemma \ref{c31c}, (1) follows.\\
\\
$(4)\Rightarrow (3)$: Let $R$ be a locally conducive domain and  $M$ a maximal ideal of $R$. Then $R_{M}$ is a Mori conducive domain, so $(R_{M})^{*}$ has only two overrings: $(R_{M})^{*}$ and $K$ \cite[Proposition 4.3]{DF}. Hence $(R_{M})^{*}$ is a one-dimensional valuation domain by \cite[Theorem 19.6]{G}, and $R_{M}\subseteq (R_{M})^{*}=(R^{*})_{R\setminus M}$ is a unibranched extension of one-dimensional domains. Hence there exists unique maximal ideal $N$ of $R^{*}$ that contracts to $M$, and (3) follows.
\end{proof}

\cite[Lemma 5.1]{adk} gives a useful upper bound of the AF-dimension of a quasilocal conducive domain with a discrete one-dimensional valuation overring. In the next lemma, we present a result that works for a different class of integral domains.

\begin{lemm}
\label{l58}
Let $R$ be a finite-absorbing quasilocal  one-dimensional domain with maximal ideal $M$.
\begin{enumerate}
    \item If $I$ is an ideal of $R$ and $n\in\mathbb{N}$, then $\omega_{R}(I)\le n$ if and only if $M^{n}\subseteq I$.
    \item Suppose that $M^2=aM$ for some $a\in M$, and $n=\max\{\omega_{R}(I)\mid I\textnormal{ is an ideal of $R$ such that }  I\not\subseteq M^2\}$ for some $n\in\mathbb{N}$. Then $\textnormal{AF-dim}(R)=n$.
\end{enumerate}
\end{lemm}

\begin{proof}
(1): The statement follows from \cite[Lemma 4]{c211}. \\
\\
(2): Assume first that $n=1$. If $I$ is an ideal of $R$ such that $M^2\subseteq I\subseteq M$, then either $I=M^2$ or $I=M$. Therefore $M^2=aM\subsetneq aR\subseteq M$ implies that $aR=M$. In other words, $R$ is Noetherian by Cohen's theorem, and must be a Dedekind domain \cite[Theorem 38.1]{G}. Thus $\textnormal{AF-dim}(R)=1$, and we are done. 

Suppose that $n\ge 2$. Since there exists an ideal $I$ of $R$ such that $I\not\subseteq M^2$ and $\omega_{R}(I)=n$ by our assumption, it follows that $\textnormal{AF-dim}(R)\ge n$. Now choose an ideal $I_{0}$ of $R$ contained in $M^2$. Since $M^2=aM$, we have $I_{0}\subseteq aM$ and $I_{0}=aI_{1}$ for some proper ideal $I_{1}$ of $R$. We also have $\omega_{R}(I_{0})>\omega_{R}(I_{1})$. Indeed, if $\omega_{R}(I_{0})=m$ for some $m\in\mathbb{N}$ (such $m$ exists since $R$ is finite-absorbing), then $M^{m}=aM^{m-1}$, so $M^{m-1}\subseteq I_{1}$ and $\omega_{R}(I_{1})\le m-1$ by (1). Now, either $I_{1}\not\subseteq M^2$ and $I_{1}$ is an $n$-absorbing ideal of $R$, or $I_{1}=aI_{2}$ for some proper ideal $I_{2}$ of $R$ with $\omega_{R}(I_{1})>\omega_{R}(I_{2})$.
Iterating this process, we deduce that $I_{0}=a^{r}I_{r}$ for some $r\in\mathbb{N}$ and an ideal $I_{r}$ of $R$ that is $n$-absorbing.
Since $M^2=aM\subseteq aR$, we have $\omega_{R}(aR)\le 2\le n$ by (1). Therefore
$I_{0}$ is a finite product of $n$-absorbing ideals of $R$, and $\textnormal{AF-dim}(R)\le n$. 
\end{proof}

\begin{theorem}
\label{af3}
\label{c52}
\begin{enumerate}
\item Let $R$ be a one-dimensional APVD such that $M$ is a principal ideal of $V$. Then $\textnormal{AF-dim}(R)\le 3$. 
\item Every Mori LAPVD has AF-dimension at most 3.
\item Let $R$ be a Mori domain. Then we have the following.
       \[
    \textnormal{AF-dim}(R)= 
\begin{cases}
    1& \text{if and only if } R \text{ is a Pr{\"u}fer domain},\\
    2             & \text{if and only if } R \text{ is an LPVD that is not a Pr{\"u}fer domain},\\
     3             & \text{if } R \text{ is an LAPVD that is not an LPVD}.
\end{cases}
\]
\end{enumerate}
\end{theorem}

\begin{proof}
(1): Let $M=aV$ for some $a\in V$. Note that $V$ is one-dimensional by Lemma \ref{l13} and $a\in M$. We also have $\bigcap\limits_{i=1}^{\infty}M^{i}=(0)$ \cite[Theorem 17.1]{G}. Hence given a proper ideal $I$ of $R$, either $IV=(0)$ or $M^{i}\subseteq IV$ for some $i$. The former yields that $I=(0)$, which is a prime ideal. The latter gives $M^{i+1}\subseteq IVM=IM\subseteq I$. Note that for each nonzero proper ideal $J$ of $R$, $\omega(J)\le n$ if and only if $M^{n}\subseteq J$ \cite[Lemma 4]{c211}. Hence we must have $\omega_{R}(I)\le i+1$, and $R$ is finite-absorbing. If $I\not\subseteq M^2$, then $IV\not\subseteq M^2$ and $M^2\subsetneq IV$, so $M^3\subseteq IM\subseteq I$ and $\omega_{R}(I)\le 3$. Therefore $\textnormal{AF-dim}(R)\le 3$ by Lemma \ref{l58}. \\
\\
(2): Let $R$ be a Mori LAPVD. We may assume that $R$ is not a field. Since every Mori APVD has Krull dimension at most 1 by Corollary \ref{moria}, $R$ is one-dimensional. Then by Theorem \ref{theo34}.(4), $R$ is of finite character. Hence by Lemma \ref{c31c}, we may assume that $R$ is a Mori APVD. Then $V=M:M$ is a DVR by Corollary \ref{moria}, so $R$ is one-dimensional and $M$ is a principal ideal of $V$, and $\textnormal{AF-dim}(R)\le3$ by (1). 
\\
\\
(3): We may assume that $R$ is not a field. Since $R$ is a Mori domain, $R$ is Pr{\"u}fer if and only if $R$ is Dedekind \cite[Corollary 2.6.21]{e19}, from which the first case follows. The second and third assertions then follow from (2) and Proposition \ref{tafmori}.
\end{proof}

Note that given $n\in\mathbb{N}$, one can construct a commutative ring with zero divisors whose AF-dimension equals $n$ \cite[Proposition 3.8]{adk}. However, every FAF-domain we mentioned so far, including the examples in \cite{adk}, is an LAPVD that has AF-dimension at most 3. Hence one may ask whether we can construct an integral domain that has AF-dimension $n$, where $n\ge 4$ is a preassigned natural number. In the next corollary we prove that such construction is possible, and the integral domain can be chosen to be non-Noetherian.

\begin{coro}
\label{c53}
Choose a field extension $F\subseteq L$, an indeterminate $X$ and $n\in\mathbb{N}$. Let 
\begin{align*}
    R&=F+X^2 L+X^4 L+\cdots+X^{2n}L+X^{2n+2}L[X],\\
S&=F+X^3 L+X^6 L+\cdots+X^{3n-3}L+X^{3n}L+X^{3n+2}L[X].
\end{align*} Then
\begin{enumerate}
    \item $R$ and $S$ are Mori FAF-domains
    .
    \item $\textnormal{AF-dim}(R)=2n+3$ and $\textnormal{AF-dim}(S)=2n+2$.
\end{enumerate}
\end{coro}

\begin{proof}
We will prove the Corollary only for $R$, since the proof can be easily adapted for the case of $S$.

Let $I=X^{2n+2}L[X]$ and $P=X^2 L+X^4 L+\cdots+X^{2n}L+X^{2n+2}L[X]$. Then $R$ is a pullback domain $R/I\times_{L[X]/I} L[X]$, and  $\textnormal{Spec}(R)=\{P, 0\}\cup\{Q\cap R\mid Q\in \textnormal{Spec}(L[X]), X\not\in Q\}$ \cite[Lemma 1.1.4.(3)]{fhp}. It follows that $R\subseteq L[X]$ is a unibranched extension, so $R$ is a one-dimensional domain of finite character.
Moreover, for each $Q\in \textnormal{Spec}(L[X])$ such that $X\not\in Q$, $R_{Q\cap R}\cong L[X]_{Q}$ is a DVR \cite[Lemma 1.1.4.(3)]{fhp} and has AF-dimension 1. Thus $ \textnormal{AF-dim}(R)= \textnormal{AF-dim}(R_{P})$ by \cite[Corollary 31]{c211}. 
Let  $V=L[X]_{XL[X]}$ and $M=PR_{P}=X^2 L+X^4 L+\cdots+X^{2n}L+X^{2n+2}V$. Then $M$ is the maximal ideal of $R_{P}$, and $M^2=X^{2}M$. It also follows that $R_{P}$ is a conducive domain by Theorem \ref{vc1}, since $X^{2n+2}\in R_{P}:V$. So $R$ is locally conducive.\\
\\
(1): Note that $X^{2}\in R:L[X]=R:R^{*}$, so by Theorem \ref{faf1} it remains to show that $R$ is Mori. In fact, by (1) and (4) of Theorem \ref{theo34}, we only need to prove that for each nonzero ideal $I$ of $R_{P}$, there exists a finitely generated ideal $J$ of $R_{P}$ such that $R_{P}:I=R_{P}:J$ and $J\subseteq I$. 
Let $I$ be a nonzero ideal of $R_{P}$, and let $T=L+M$. Then $T$ is a quasilocal overring of $R_{P}$ with maximal ideal $M$. Moreover, $T$ is a Noetherian ring by Eakin-Nagata theorem, so $IT=JT$ for some finitely generated ideal $J\subseteq I$. Since $R_{P}:T=X^{2}R_{P}$, we have $R_{P}:HT=(R_{P}:T):H=X^{2}(R_{P}:H)$ for each ideal $H$ of $R_{P}$. Therefore $R_{P}:I=X^{-2}(R_{P}:IT)=X^{-2}(R_{P}:JT)=R_{P}:J$.
\\
\\
(2): Let $\mathcal{S}=\{I\textnormal{ is an ideal of }R_{P}\mid   I\not\subseteq M^2\}$ and choose $I\in\mathcal{S}$. Then we have $IV=X^{r}V$ for some $r\le 2n+3$. Hence $4n+4+i-r\ge 2n+2$ for each $i\in\mathbb{N}$, and $X^{4n+4+i}L\subseteq fR$ for every $f\in V$ with $fV=X^{r}V$. So  $M^{2n+3}=X^{4n+4}M=X^{4n+6} L+X^{4n+8} L+\cdots+X^{6n+4}L+X^{6n+6}V\subseteq I$, and $\omega_{R_{P}}(I)\le 2n+3$.  On the other hand, let $I=X^{2n+3}R$. Then $I\in \mathcal{S}$, and $M^{2n+2}\not\subseteq I$ since $X^{4n+4}\in M^{2n+2}\setminus I$. Thus $\omega_{R_{P}}(I)=2n+3$, and $2n+3=\max\{\omega_{R_{P}}(I)\mid I\in\mathcal{S}\}$. Moreover, $R_{P}$ is  finite-absorbing by \cite[Theorem 29.(2)]{c211}, \cite[Proposition 3.5]{adk} and (1). Therefore $\textnormal{AF-dim}(R_{P})=2n+3$ by Lemma \ref{l58}.
\end{proof}

\begin{rem}
\begin{enumerate}
    \item The ring $R$ in Corollary \ref{c53} is Noetherian exactly when $[L:F]<\infty$ \cite[Theorem 4]{br76}. 
    \item Unlike TAF-domains, FAF-domains are not necessarily Mori. Let $V$ be a valuation domain of value group $\mathbb{Q}$ that contains a field of characteristic zero \cite[Proposition 18.4, Corollary 18.5]{G}. Then for a nonzero nonunit $a\in V$, set $M=aV$. It follows that $V/M$ has characteristic zero, so it contains a field $L$ (of characteristic zero). Now let $R=L\times_{V/M}V$ 
which is an APVD with maximal ideal $M$ by Proposition \ref{pullback1}. Since $\textnormal{AF-dim}(R)\le3$ by Theorem \ref{af3}.(1), $R$ is an FAF-domain. But the value group of $V$ is a nondiscrete subgroup of $\mathbb{R}$, so $V=M:M$ is not a DVR and $R$ is not a Mori domain by Corollary \ref{moria}, and $\textnormal{AF-dim}(R)=3$ by Proposition \ref{tafmori}.
\end{enumerate}
\end{rem}

We now focus on the pullback properties of FAF-domains.

\begin{prop}
\label{p56}
Let $R$ be a Mori FAF-domain with $R:R^{*}\neq (0)$.
Then there exists a Dedekind domain $T$, maximal ideals $N_{1},\dots, N_{n}$ of $T$, $a_{1},\dots, a_{n}\in\mathbb{N}$ and unibranched ring extensions $D_{i}\subseteq T/N_{i}^{a_{i}}$ such that $R=\pi^{-1}(\prod\limits_{i=1}^{n} D_{i})$ where $\pi:T\to \prod\limits_{i=1}^{n} (T/N_{i}^{a_{i}})$ is the canonical projection.
\end{prop}

\begin{proof}
 $R^{*}$ is a Dedekind domain by Theorem \ref{bh96t}.(1). Therefore, letting $T=R^{*}$ and $I=R:R^{*}$, it follows that $R=R/I\times_{T/I}T$.  Moreover, there exist only finitely many maximal ideals $N_{1},\dots, N_{n}$ of $T$ that contains $I$, and $I=\prod\limits_{i=1}^{n}N_{i}^{a_{i}}$ for some $a_{i}\in\mathbb{N}$. Fix $i\in\{1,\dots, n\}$ and let $M_{i}=N_{i}\cap R$, $D_{i}=R/(N_{i}^{a_{i}}\cap R)$. Then  $D_{i}\subseteq T/N_{i}^{a_{i}}$ is a unibranched ring extension, and $R=\pi^{-1}(\prod\limits_{i=1}^{n} D_{i})$\cite[Lemma 1.1.6]{fhp}.
\end{proof}

\begin{prop}
\label{fafp}
Let $R$ be an integral domain that is not a field. Then the following are equivalent.
\begin{enumerate}
    \item $R$ is a Noetherian FAF-domain with $R:R'\neq (0)$.
    \item There exists a Dedekind domain $T$, maximal ideals $N_{1},\dots, N_{n}$ of $T$, $a_{1},\dots, a_{n}\in\mathbb{N}$ and unibranched ring extensions $D_{i}\subseteq T/N_{i}^{a_{i}}$ such that 
     $T/N_{i}^{a_{i}}$ is a finite $D_{i}$-module for each $i$, and $R=\pi^{-1}(\prod\limits_{i=1}^{n} D_{i})$ where $\pi:T\to \prod\limits_{i=1}^{n} (T/N_{i}^{a_{i}})$ is the canonical projection.
     \item $R'$ is a Dedekind domain, and $R/(R:R')\to R'/(R:R')$ is a unibranched ring extension that is also a finite module extension. 
\end{enumerate}
In particular, if $(1)$ holds, then $R'$ is a finite $R$-module
and $\textnormal{AF-dim}(R)=\max\limits_{1\le i\le n}\{\textnormal{AF-dim}(R_{N_{i}\cap R})\}$.
\end{prop}

\begin{proof}
(1)$\Rightarrow$(2): Suppose that $R$ is a Noetherian FAF-domain with $R:R'\neq (0)$. By Proposition \ref{p56} there exists a Dedekind domain $T$, maximal ideals $N_{1},\dots, N_{n}$ of $T$, $a_{1},\dots, a_{n}\in\mathbb{N}$ and  $D_{i}\subseteq T/N_{i}^{a_{i}}$ is a unibranched ring extension for each $i\in\{1,\dots, n\}$ such that $R=\pi^{-1}(\prod\limits_{i=1}^{n} D_{i})$ where $\pi:T\to \prod\limits_{i=1}^{n} (T/N_{i}^{a_{i}})$ is the canonical projection. Therefore, $R_{M_{i}}=D_{i}\times_{T/N_{i}^{a_{i}}}T_{N_{i}}$, where $M_{i}=N_{i}\cap R$ for each $i\in\{1,\dots, n\}$ \cite[Lemma 1.1.6]{fhp}. Moreover, $R_{M_{i}}$ is a Noetherian conducive domain by Theorem \ref{faf1}. Hence $T/N_{i}^{a_{i}}$ is a finite $D_{i}$-module \cite[Theorem 6]{bdf86}. 
\\
\\
(2)$\Rightarrow$(3): Suppose that $(2)$ holds. For each $i\in\{1,\dots, n\}$, since $D_{i}\to T/N_{i}^{a_{i}}$ is unibranched and finite, so is $\prod\limits_{i=1}^{n}D_{i}\subseteq \prod\limits_{i=1}^{n}(T/N_{i}^{a_{i}})\cong T/\prod\limits_{i=1}^{n}N_{i}^{a_{i}}$. It follows that $\prod\limits_{i=1}^{n}N_{i}^{a_{i}}$ is a common ideal of $R$ and $T$, so $R'=R^{*}=T^{*}=T$ and $R:R'=\prod\limits_{i=1}^{n}N_{i}^{a_{i}}$. Hence $(3)$ follows.\\
\\
(3)$\Rightarrow$(1): Assume (3). If $R:R'=(0)$, then $R'$ is a finite $R$-module, which is a contradiction. Therefore $R:R'\neq (0)$, and $R$ is a pullback domain $R/(R:R')\times_{R'/(R:R')}R'$. Therefore, $R$ is Noetherian \cite[Proposition 1.1.7]{fhp} and $R\subseteq R'$ is unibranched \cite[Lemma 1.1.4.(3)]{fhp}. It follows that $R$ is an FAF-domain by Theorem \ref{faf1}.
\end{proof}

\begin{prop}
\label{lpvd58}
Let $R$ be an integral domain. Then the following are equivalent.
\begin{enumerate}
    \item $R$ is a Mori LAPVD with $R:R^{*}\neq (0)$.
    \item There exists a Dedekind domain $T$, maximal ideals $N_{1},\dots, N_{n}$ of $T$, $a_{1},\dots, a_{n}\in\mathbb{N}$ and subfields  $F_{i}$ of $T/N_{i}^{a_{i}}$ for each $i\in\{1,\dots, n\}$, such that $R=\pi^{-1}(\prod\limits_{i=1}^{n} F_{i})$ where $\pi:T\to \prod\limits_{i=1}^{n} (T/N_{i}^{a_{i}})$ is the canonical projection.
    \item $R^{*}$ is a Dedekind domain, $R/(R:R^{*})$ is an Artinian reduced ring and $R/(R:R^{*})\subseteq R^{*}/(R:R^{*})$ is a unibranched ring extension.
\end{enumerate}
\end{prop}

\begin{proof}
(1)$\Rightarrow$(2): Suppose that (1) holds. By Proposition \ref{p56}, there exists a Dedekind domain $T$, maximal ideals $N_{1},\dots, N_{n}$ of $T$, $a_{1},\dots, a_{n}\in\mathbb{N}$ and unibranched ring extensions $D_{i}\subseteq T/N_{i}^{a_{i}}$ such that $R=\pi^{-1}(\prod\limits_{i=1}^{n} D_{i})$ where $\pi:T\to \prod\limits_{i=1}^{n} (T/N_{i}^{a_{i}})$ is the canonical projection. Since $R_{M_{i}}=D_{i}\times_{T/N_{i}^{a_{i}}}T_{N_{i}}$, $D_{i}$ must be a field by Proposition \ref{pullback1}.\\
\\
(2)$\Rightarrow$(3): If $(2)$ holds, then let $M_{i}=N_{i}\cap R$ for each $i\in\{1,\dots, n\}$. Since $R:R^{*}=N_{1}^{a_{1}}\cdots N_{n}^{a_{n}}$, $\{M_{i}\}_{i=1}^{n}$ is the set of (distinct) prime ideal of $R$ that contains $R:R^{*}$ \cite[Lemma 1.1.4.(6)]{fhp}. It then follows that $R:R^{*}=M_{1}\cap\cdots \cap M_{n}$, and $R/(R:R^{*})$ is an Artinian reduced ring and $R/(R:R^{*})\subseteq R^{*}/(R:R^{*})$ is a unibranched ring extension. By \cite[Lemma 1.1.4.(10)]{fhp}, $R^{*}=T^{*}=T$ is Dedekind.\\
\\
(3)$\Rightarrow$(1): Assume that (3) is true. It is routine to see that $R:R^{*}\neq (0)$.  Since $R$ is the pullback domain $R/(R:R^{*})\times_{R^{*}/(R:R^{*})}R^{*}$, it follows that $R\subseteq R^{*}$ is unibranched \cite[1.1.4.(3)]{fhp}. It also follows that $R$ is an LAPVD by \cite[Lemma 1.1.6]{fhp} and Proposition \ref{pullback1}. By Corollary \ref{triv58}, we obtain the result promised.
\end{proof}

In \cite[Corollary 3.9]{adk}, the authors gave a complete description of AF-dimension of $\mathbb{Z}[\sqrt{m}]$ when $m$ is a square-free integer. On the other hand, \cite[Theorem 2.5]{df87} establishes a characterization theorem of an order of a quadratic number fields being a GPVD. In the last topic of this section, we extend these results simultaneously.

Recall that given a square-free integer $n$, a rank 2 free $\mathbb{Z}$-submodule of a quadratic number field $\mathbb{Q}(\sqrt{n})$ is called an \textit{order} of $\mathbb{Q}(\sqrt{n})$. It is well-known that every quadratic order of $\mathbb{Q}(\sqrt{n})$ is of the form $\mathbb{Z}[r\omega_{n}]$ for some $r\in\mathbb{N}$ where 
\[\omega_{n}=
\begin{cases}\cfrac{1+\sqrt{n}}{2} &\text{ if } n\equiv 1(\textnormal{mod } 4),\\
\\
\sqrt{n} &\text{ if } n\not\equiv 1(\textnormal{mod } 4).
\end{cases}
\]
Recall that for an integer $a$ and a prime number $p$, the Kronecker symbol $a \choose p$ is defined as follows:
\[{a \choose 2}=
\begin{cases}
0 &\text{ if } a\equiv 0(\textnormal{mod } 2)\\
1 & \text{ if } a\equiv \pm1(\textnormal{mod } 8)\\
-1 & \text{ if } a\equiv \pm3(\textnormal{mod } 8)
\end{cases}
\]
when $p\neq 2$, $a\choose p$ equals the Legendre symbol:
\[{a \choose p}=
\begin{cases}
0 &\text{ if } p\mid a\\
1 & \text{ if } p\nmid a \text{ and }  x^2\equiv a(\textnormal{mod } p) \text{ has an integer solution}\\
-1 & \text{ otherwise }
\end{cases}
\]
For a square-free integer $n\not\in\{0,1\}$, 
 \[d_{n}=
\begin{cases}
n & \text{ if } n\equiv 1(\textnormal{mod } 4)\\
4n & \text{ otherwise }
\end{cases}
\]
is called the \textit{discriminant} of  $\mathbb{Q}(\sqrt{n})$.
\begin{theorem}
\label{cox}
\cite[Proposition 5.16]{c89}
Let $n$ be a square-free integer and $p$ a prime number. Then
\begin{enumerate}
    \item If ${d_{n} \choose p}=0$, then $p\mathbb{Z}[\omega_{n}]=P^2$ for some prime ideal $P$ of  $\mathbb{Z}[\omega_{n}]$.
    \item If ${d_{n} \choose p}=1$, then $p\mathbb{Z}[\omega_{n}]=P_{1}P_{2}$ for some distinct prime ideals $P_{1}, P_{2}$ of  $\mathbb{Z}[\omega_{n}]$.
    \item If ${d_{n} \choose p}=-1$, then $p\mathbb{Z}[\omega_{n}]$ is a prime ideal of  $\mathbb{Z}[\omega_{n}]$.
\end{enumerate}
\end{theorem}

\begin{theorem}
\label{quad}
Let $n$ be a square-free integer, $r\in\mathbb{N}\setminus\{1\}$ with prime factorization $r=p_{1}^{a_{1}}\cdots p_{m}^{a_{m}}$, and
$R=\mathbb{Z}[r\omega_{n}]$. Then
\begin{enumerate}
    \item $\textnormal{AF-dim}(R)<\infty$ if and only if ${d_{n} \choose p_{i}}\neq 1$ for each $i\in\{1,\dots, m\}$. 
\item Suppose that $\textnormal{AF-dim}(R)<\infty$, and set\begin{align*}
    \mathcal{F}_{1}&=\{i\in\{1,\dots, m\}\mid {d_{n} \choose p_{i}}=0\},\\
    \mathcal{F}_{2}&=\{i\in\{1,\dots, m\}\mid {d_{n} \choose p_{i}}=-1\}.
\end{align*} Then $\textnormal{AF-dim}(R)=
\max\{\max\limits_{i\in\mathcal{F}_{1}}\{2a_{i}+1\}, \max\limits_{i\in\mathcal{F}_{2}}\{2a_{i}\}\}.
$\end{enumerate}
\end{theorem}


\begin{proof}
(1): Let $I=R:R'$. Note that $R'=\mathbb{Z}[\omega_{n}]$ is a Dedekind domain. Thus we have the pullback\[ \begin{tikzcd}
R \arrow{r} \arrow[swap]{d} & 
R/I \cong \mathbb{Z}/r\mathbb{Z}\arrow{d}{\iota} \\%
R'\arrow{r}{\pi}& R'/I\cong R'/rR'.
\end{tikzcd}
\]
 It follows that $R$ is an FAF-domain if and only if the number of minimal prime ideals of $I$ in $R$ is the same as that of $I$ in $R'$ (Proposition \ref{fafp}), which happens exactly when $pR'$ is a power of a prime ideal of $R'$ for each prime factor $p$ of $r$ in $\mathbb{Z}$. Theorem \ref{cox} then yields (1).\\
\\
(2): Note that $\{1,\dots, m\}=\mathcal{F}_{1}\cup \mathcal{F}_{2}$ by (1). Fix $j\in\{1,\dots, m\}$, and let $N$ be the maximal ideal of $R'$ that contains $p_{j}R'$. Let $V=R'_{N}$, $S=R_{N\cap R}$ and $M$ the maximal ideal of $S$. Then $V$ is a DVR, and  $M^2=p_{j}M$ since $M=p_{j}S+p_{j}^{a_{j}}V$. Choose $b\in N$ such that $NV=bV$.  Suppose that $j\in\mathcal{F}_{1}$. Then $p_{j}R'=N^2$ by Theorem \ref{cox}, so $N^{2a_{j}+2}V=p_{j}^{a_{j}+1}V\subseteq p_{j}M=M^2$. Hence if $I$ is an ideal of $S$ such that $I\not\subseteq M^2$, then  $IV=fV$ for some $f\in I$ such that $fV=N^{t}V$ for some $t\le 2a_{j}+1$. In other words, $f=b^{t}u$ for some unit $u$ of $V$.  Then $p_{j}^{2a_{j}+1}\in fp_{j}^{a_{j}}V$
and $p_{j}^{3a_{j}+1}V\subseteq fp_{j}^{a_{j}}V$, so $M^{2a_{j}+1}=p_{j}^{2a_{j}}M\subseteq fM\subseteq I$. On the other hand, $J=p_{j}^{a_{j}}bS$ is an ideal of $S$ such that $J\not\subseteq M^2$, and $p_{j}^{2a_{j}}\in M^{2a_{j}}\setminus J$. Thus by Lemma \ref{l58}, $\textnormal{AF-dim}(R_{N\cap R})=2a_{j}+1$.

Now consider the case when $j\in \mathcal{F}_{2}$. Then $N=p_{j}R'$ by Theorem \ref{cox}. We have  $N^{a_{j}+1}V=p_{j}^{a_{j}+1}V\subseteq p_{j}M=M^2$. Hence if $I$ is an ideal of $S$ such that $I\not\subseteq M^2$, then  $IV=p_{j}^{s}V$ for some $s\le a_{j}$. In other words, $f=b^{s}v$ for some unit $v$ of $V$. Since $p_{j}^{2a_{j}}\subseteq fp_{j}^{a_{j}}V$ and $p_{j}^{3a_{j}-1}\subseteq fp_{j}^{a_{j}}V$,  
 it follows that $M^{2a_{j}}=p_{j}^{2a_{j}-1}M\subseteq I$.  We then claim that $V\neq S+N$. Indeed, if $V=S+N$, then $M=p_{j}S+p_{j}^{a_{j}}V=p_{j}S+p_{j}^{a_{j}}(S+N)=p^{j}S+p^{a_{j}}S+p^{a_{j}}N=p_{j}S$, so $S$ is a DVR by Krull intersection theorem. Then  $S=V$, which is a contradiction. Hence, $V\neq S+N$ and there exists a unit $w$ of $V$ such that $w\not\in S+N$. Note also that $w^{-1}\not\in S+N$ since $S+N$ is an integral overring of $S$. Now $J'=p_{j}^{a_{j}}wS$ is an ideal of $S$ not in $M^2$. Indeed, if $p_{j}^{a_{j}}w\in M^2$, then $p_{j}^{a_{j}}w=p_{j}^2s+p_{j}^{a_{j}+1}u$ for some $s\in S, u\in V$. We may assume that $s=p^{a}u'$ for some $a\in\mathbb{N}_{0}$ and $u'\in R$, where $u'$ is a unit of $R$. Then $p_{j}^{a_{j}}w-p_{j}^2s=p_{j}^{a_{j}+1}u$. Since $p_{j}$ generate the maximal ideal of $V$, we must have $a=a_{j}-2$ and $w=u'+p_{j}u\in S+N$, a contradiction. Similarly, we deduce that $p_{j}^{a_{j}-1}w^{-1}\not\in S$, so $p_{j}^{2a_{j}-1}\in M^{2a_{j}-1}\setminus J'$. Therefore, we have $\textnormal{AF-dim}(R_{N\cap R})=2a_{j}$ when $j\in\mathcal{F}_{2}$ by Lemma \ref{l58}. The conclusion now follows from Proposition \ref{fafp} and the pullback structure of $R$ discussed in the proof of (1) of this theorem.
\end{proof}

The first part of the following corollary retrieves \cite[Theorem 2.5]{df87}.

\begin{coro}
\label{simp8}
Let $n\in\mathbb{Z}\setminus\{0\}$ be square-free, and $R=\mathbb{Z}[r\omega_{n}]$ for some $r\in\mathbb{N}$. 
\begin{enumerate}
    \item The following are equivalent.
\begin{enumerate}
    \item $R$ is a GPVD.
    \item $R$ is an LPVD.
    \item $\textnormal{AF-dim}(R)\le 2$.
    \item $r$ is square-free and ${d_{n}\choose p}=-1$ whenever $p$ is a prime factor of $r$.
\end{enumerate}
    \item The following are equivalent.
\begin{enumerate}
    \item $R$ is an LAPVD.
    \item $\textnormal{AF-dim}(R)\le 3$.
    \item $r$ is square-free and ${d_{n}\choose p}\neq 1$ whenever $p$ is a prime factor of $r$.
\end{enumerate}
 \item The following are equivalent.
\begin{enumerate}
    \item $R$ is locally conducive.
    \item $\textnormal{AF-dim}(R)<\infty$.
    \item ${d_{n}\choose p}\neq 1$ whenever $p$ is a prime factor of $r$.
\end{enumerate}
\end{enumerate}
\end{coro}

\begin{proof}
(1): Since $R$ is Noetherian and $R:R'\neq (0)$, $(a)\Rightarrow (b)$ follows from Lemma \ref{66}, and
$(b)\Rightarrow(c)$ from Proposition \ref{tafmori}. The equivalence of $(c)$ and $(d)$ follows from Theorem \ref{quad} and \cite[Corollary 3.9]{adk}.  $(d)\Rightarrow(a)$ can be deduced from Theorem \ref{tafpullback} and Theorem \ref{cox}.\\
\\
(2): $(a)\Rightarrow (b)$ follows from Theorem \ref{af3}.(2). The equivalence of $(b)$ and $(c)$ follows from Theorem \ref{quad} and \cite[Corollary 3.9]{adk}.   $(c)\Rightarrow(a)$ can be deduced from Proposition \ref{lpvd58} and Theorem \ref{cox}.\\
\\
(3): Follows from Theorems \ref{faf1} and \ref{quad}.(1).
\end{proof}

\section{Rings of the form $A+XB[X]$}
How the polynomial ring $R[X]$ behaves when we manipulate the coefficient ring $R$ has been a stimulating topic for ring theorists. Probably the most famous result in this direction is the celebrated Hilbert basis theorem which states that if $R$ is a Noetherian ring, then so is $R[X]$. In this spirit, several researchers studied the structure of the rings of the form $A+XB[X]$ where $A\subseteq B$ is an extension of rings (see \cite{z03} and its reference list), and one of the main topics was the investigation of Krull dimension of $A+XB[X]$ under the assumption that $A$ and $B$ are integral domains. In this section, we take an opposite approach by restricting the Krull dimension of $A+XB[X]$ to one and studying the behavior of $A$ and $B$. Specifically, we characterize when $A+XB[X]$ is a TAF-ring in terms of $A$ and $B$. 
We first need the following well-known lemma.

\begin{lemm}
\label{kzero}
\cite[Exercise 1-6.1]{K}
 Let $A\subseteq B$ an extension of rings and $P$ a minimal prime ideal of $A$. Then $N\cap A=P$ for some prime ideal $N$ of $B$.
\end{lemm}

\begin{proof}
Let $S=A\setminus P$. Then $S$ is a multiplicatively closed subset of $B$. Choose an ideal $N$ maximal with respect to the property such that $N\cap S=\emptyset$. Such $N$ is a prime ideal of $B$, so $N\cap A$ is a prime ideal of $A$ contained in $P$. Since $P$ is a minimal prime ideal of $A$, we have $N\cap A=P$.
\end{proof}
 
In the following lemma, a couple of well-known results concerning the prime ideals of rings of the form $A+XB[X]$ are collected. Recall that when we say $\textnormal{Spec}(R)$ is \textit{Noetherian} for a ring $R$, we mean that $R$ satisfies the ascending chain condition on radical ideals.
 
 \begin{lemm}
 \label{con56}
Let $A\subseteq B$ an extension of rings and $R=A+XB[X]$. Then 
\begin{enumerate}
    \item  $\textnormal{Spec}(R)=S_{1}\cup S_{2}$, where \begin{align*}
    S_{1}&=\{P+XB[X]\mid P\in \textnormal{Spec}(A)\},\\
    S_{2}&=\{Q\cap R\mid Q\in \textnormal{Spec}(B[X]), X\not\in Q\}.
\end{align*}
\item $\textnormal{Spec}(R)$ is Noetherian if and only if $\textnormal{Spec}(A)$ and $\textnormal{Spec}(B)$ are Noetherian.
\item $1+\max\{\textnormal{dim}(A), \textnormal{dim}(B)\}\le\textnormal{dim}(R)\le \textnormal{dim}(A)+\textnormal{dim}(B[X])$.
\end{enumerate}
 \end{lemm}
 
 \begin{proof}
(1): Since $R$ is a pullback ring $A\times_{B[X]/XB[X]}B[X]$, (1) holds by \cite[Lemma 1.1.4]{fhp}.\\
\\
(2): \cite[Proposition 6.1.(2)]{p97}.\\
\\
(3): Let $P\in\textnormal{Spec}(A)$ and $M\in\textnormal{Spec}(B)$. Choose a minimal prime ideal $P_{0}$ of $A$ contained in $P$, and $N\in\textnormal{Spec}(B)$ such that $P_{0}=N\cap A$ (such $N$ exists by Lemma \ref{kzero}). Then $N[X]\cap R=P_{0}+XN[X]$ is a minimal prime ideal of $R$ properly contained in $P+XB[X]$. Hence $1+ht_{A}(P)\le ht_{R}(P+XB[X])\le\textnormal{dim}(R)$. On the other hand, $M[X]\cap R=(M\cap A)+XM[X]\subsetneq (M\cap A)+XB[X]$, so by (1) $1+ht_{B}(M)\le 1+ht_{R}(M[X]\cap R)\le\textnormal{dim}(R)$. Thus the first inequality of (3) follows. The second inequality follows from the fact that $\textnormal{Spec}(R)$ is a quotient space of the disjoint union of $\textnormal{Spec}(A)$ and $\textnormal{Spec}(B[X])$ \cite[Theorem 1.4]{fhp}.
 \end{proof}
 
Throughout this manuscript, whenever $R$ is of the form $A+XB[X]$ for a ring extension $A\subseteq B$, the notations $S_{1}$ and $S_{2}$ will be used to denote the sets introduced in Lemma \ref{con56}.(1).  The following theorem, taken from \cite[Corollaire 9]{r81}, is crucial in the proof of Lemma \ref{keylemma}.

\begin{theorem}
\label{rs2}
Let $A\subseteq B$ be an extension of rings such that $B$ is a strongly Laskerian ring. Suppose that there exists a nonzero common ideal $I$ of $A$ and $B$ contained in only finitely many prime ideals of $A$, and these prime ideals are all maximal ideals of $A$.
Then $A$ is a strongly Laskerian ring.
\end{theorem}
 
 \begin{lemm}
 \label{keylemma} 
Let $A\subseteq B$ be an extension of rings and $R=A+XB[X]$. Then $R$ is a one-dimensional strongly Laskerian ring if and only if $A$ is a zero-dimensional strongly Laskerian ring and $B$ is an Artinian ring. 
 \end{lemm}
 
\begin{proof}
Suppose that $R$ is one-dimensional strongly Laskerian ring. Then both $A$ and $B$ are zero-dimensional rings by Lemma \ref{con56}.(3). Since $A\cong R/XB[X]$ and $R$ is strongly Laskerian, so is $A$. It then remains to show that $B$ is Noetherian. We will use the argument similar to that of  \cite[Proposition 59]{c211}.  Suppose that $B$ is not Noetherian, and choose  $\{b_{n}\}_{n\in\mathbb{N}}\subseteq B$ so $\{(b_{1},\dots, b_{n})B\}_{n\in\mathbb{N}}$ is a strictly ascending chain of ideals of $B$. Set $J$ be the ideal of $R$ generated by $\{b_{n}X^{n}\}_{n\in\mathbb{N}}$. Then for each $n\in\mathbb{N}\setminus\{1\}$,
$b_{n}X^{n}=(b_{n}X)X^{n-1}$ is a product of $n$ elements of $R$, but no $(n-1)$-subproduct of it is in $J$. Thus $\omega_{R}(J)\ge n-1$. Since $n$ is chosen arbitrarily, we must have $\omega_{R}(J)=\infty$. However, since $R$ is strongly Laskerian, it must be finite-absorbing \cite[Lemma 19]{c211}, so we have a contradiction. Therefore $B$ must be Noetherian.

Conversely, suppose that $A$ is a zero-dimensional strongly Laskerian ring and $B$ is an Artinian ring. Then $B[X]$ is a Noetherian ring, and $XB[X]$ is a nonzero common ideal of $R$ and $B[X]$. From Lemma \ref{con56}.(1) it also follows that every prime ideal of $R$ containing $XB[X]$ is a maximal ideal, and there are only finitely many minimal prime ideals of $XB[X]$ in $R$. Hence we conclude that $XB[X]$ is contained in only finitely many prime ideals of $R$. Thus $R$ is strongly Laskerian by Theorem \ref{rs2}, and it is one-dimensional by Lemma \ref{con56}.(3). 
\end{proof}

Recall that a commutative ring $R$ is said to be \textit{von Neumann regular} if it is reduced and its Krull dimension is zero, \textit{B{\'e}zout} if every finitely generated ideal of $R$ is principally generated, \textit{arithmetical} if the set of ideals of $R_{M}$ forms a chain under set inclusion for each maximal ideal $M$ of $R$, and \textit{satisfies $(*)$} if each ideal of $R$ whose radical is prime is a primary ideal of $R$ \cite{g62, g64, gm65}. 
It is known that $R[X]$ is a B{\'e}zout ring if and only if $R$ is von Neumann regular \cite[Theorem 18.7]{g84}, \cite[Theorem 6]{a76}. If $A\subsetneq B$ is a ring extension, then
even if $A$ and $B$ are both von Neumann regular rings, $A+XB[X]$ may not be B{\'e}zout; if $A\subsetneq B$ are fields, then $A+XB[X]$ is not B{\'e}zout \cite[Corollary 1.1.9.(1)]{fhp}. On the other hand, $\mathbb{Z}+X\mathbb{Q}[X]$, being a pullback domain $\mathbb{Z}\times_{\mathbb{Q}}\mathbb{Q}[X]$, is a two-dimensional B{\'e}zout domain by Lemma \ref{con56}.(3) and \cite[Corollary 1.1.11]{fhp}, 
 while $R[X]$ is a one-dimensional ring if $R[X]$ is B{\'e}zout. Hence one may conjecture that $A+XB[X]$ is a one-dimensional B{\'e}zout ring if and only if $A=B$ is a von Neumann regular ring, and this is indeed true as the next Proposition shows.

\begin{prop}
\cite[Corollaries 3.4 and 4.8]{mad18})
\label{coro56}
Let $A\subseteq B$ rings, $X$ an indeterminate, and $R=A+XB[X]$.
\begin{enumerate}
\item The following are equivalent.
\begin{enumerate}
    \item $R$ is a one-dimensional reduced ring.
    \item Both $A$ and $B$ are von Neumann regular rings.
\end{enumerate}
\item The following are equivalent.
\begin{enumerate}
\item $R$ is a Mori PVD for each maximal ideal $M$ of $R$. 
    \item $R_{M}$ is a one-dimensional domain for each maximal ideal $M$ of $R$.
    \item $R$ satisfies  $(*)$.
    \item $R$ is a one-dimensional LPVR.
    \item $R$ is a one-dimensional locally divided ring.
    \item $A\subseteq B$ is a unibranched extension of von Neumann regular rings.
\end{enumerate}
   \item The following are equivalent.
    \begin{enumerate}
        \item $R$ is a one-dimensional strongly Laskerian reduced ring.
        \item Both $A$ and $B$ are  semi-quasilocal von Neumann regular rings.
    \end{enumerate}
   \item The following are equivalent.
    \begin{enumerate}
        \item $R$ is a one-dimensional B{\'e}zout ring.
        \item $R_{M}$ is a DVR for each maximal ideal $M$ of $R$.
        \item $R$ is a one-dimensional arithmetical ring.
        \item $A=B$ is a von Neumann regular ring.
    \end{enumerate}
    \item The following are equivalent.
\begin{enumerate}
    \item $R$ is a TAF-ring.
    \item $R$ is an FAF-ring.
    \item $R$ is strongly Laskerian and locally divided.
    \item $R$ is finite-absorbing and locally divided.
    \item $A\subseteq B$ is a unibranched extension of semi-quasilocal von Neumann regular rings.
   \item There exist fields $\{F_{i}\}_{i=1}^{n}$ and $\{L_{i}\}_{i=1}^{n}$ such that  $A\cong\prod\limits_{i=1}^{n}F_{i}$,  $B\cong\prod\limits_{i=1}^{n}L_{i}$ and $F_{i}\subseteq L_{i}$ for each $i\in\{1,\dots, n\}$.
    \end{enumerate}
    \item (cf. \cite[Corollary 4.4]{adk}) The following are equivalent.
    \begin{enumerate}
        \item $R$ is a general ZPI-ring.
        \item $R$ is a strongly Laskerian B{\'e}zout ring.
        \item $R$ is a TAF-ring and $A=B$.
        \item $A=B$ is a semi-quasilocal von Neumann regular ring.
    \end{enumerate}
\end{enumerate}
\end{prop}

\begin{proof}
(1): $(a)\Rightarrow (b)$: Suppose that $R$ is one-dimensional reduced ring. Then both $A$ and $B$ are zero-dimensional by Lemma \ref{con56}.(3). Now, $A$, being a subring of a reduced ring $R$, is a reduced ring. Similarly, if $b\in Nil(B)$, then $bX\in Nil(R)=(0)$, so $b=0$ and we conclude that $Nil(B)=(0)$. Since both $A$ and $B$ are reduced rings with Krull dimension zero, they must be von Neumann regular.\\
\\
$(b)\Rightarrow (a)$: Suppose that $(b)$ holds. Then $\textnormal{dim}(A)=0$ and $\textnormal{dim}(B[X])=1$ \cite[Corollary 30.3]{G}, so we have $\textnormal{dim}(R)=1$ by Lemma \ref{con56}.(3). On the other hand, since $B$ is reduced, so is $B[X]$. Hence $R$ is reduced since it is a subring of $B[X]$.\\
\\
(2): $(a)\Rightarrow(b)$: By Proposition \ref{tafmori}, a Mori PVD has Krull dimension at most 1. Since $R$ does not have any maximal ideal of height 0, $(a)$ implies $(b)$.\\
\\
$(b)\Rightarrow(c)\Rightarrow (e)$: Follows from the fact that $R$ satisfies $(*)$ if and only if for each maximal ideal $M$ of $R$, $R_{M}$ is either a zero-dimensional ring or a one-dimensional domain \cite[Theorem 1]{gm65}.\\
\\
$(e)\Rightarrow (f)$: Suppose that $(e)$ holds. 
Now choose a maximal ideal $M$ of $R$ and a minimal prime ideal $I$ of $R$ contained in $M$. Since $\textnormal{dim}(A)=0$ by Lemma \ref{con56}.(3), $I\not\in S_{1}$. Hence $I\in S_{2}$, and $I=N[X]\cap R$ for some $N\in\textnormal{Spec}(B)$. Let $P=N\cap A$, and we have $I=P+XN[X]$. Since $X\not\in I$, we have $Nil(R)R_{M}\subseteq IR_{M}\subseteq XR_{M}$ since $R_{M}$ is locally divided. Therefore $Nil(R)\subseteq XR$. If $a\in Nil(A)$, then $a\in Nil(R)\subseteq XR$, so we must have $a=0$. Hence $Nil(A)=(0)$ and $A$ is reduced. Similarly, if $b\in Nil(B)$, then $bX\in Nil(R)\subseteq XR=XA+X^2B[X]$, so $bX\in XA$ and $b\in A$. Thus $b\in Nil(A)=(0)$ and we conclude that $Nil(B)=(0)$. Then $R$ is reduced, being a subring of a reduced ring $B[X]$. Hence $R$ is a one-dimensional reduced ring, and both $A$ and $B$ are von Neumann regular by (1).

It remains to show that $A\subseteq B$ is unibranched. Let $P$ be a prime ideal of $A$. Then by Lemma \ref{kzero} there exists a prime ideal of $B$ that contracts to $P$. On the other hand, suppose that there exist two prime ideals $M,N$ of $B$ that contracts to $P$. Then $M'=M[X]\cap R=P+XM[X]$ and $N'=N[X]\cap R=P+XN[X]$ are incomparable prime ideals of $R$ contained in the maximal ideal $P+XB[X]$ of $R$. Since $R$ is locally divided, this is a contradiction \cite[Proposition 2.1.(d)]{bd01}. Hence there exists exactly one prime ideal of $B$ that contracts to $P$. Consequently, $A\subseteq B$ is unibranched. \\
\\
$(f)\Rightarrow (a)$: Suppose that $(f)$ holds, and choose a maximal ideal $M$ of $R$. If $M\in S_{1}$, then $M=P+XB[X]$ for some $P\in \textnormal{Spec}(A)$. Let $S=(A\setminus P)+XB[X]$ and choose $N\in\textnormal{Spec}(B)$ such that $M=N\cap A$. Since $A\subseteq B$ is unibranched, $M_{0}=N[X]\cap R=P+XN[X]$ is the only prime ideal of $R$ properly contained in $M$, and $A_{P}$ is a subfield of a field $B_{N}$ since the former is isomorphic to $A/P$ and the latter to $B/N$. Therefore, $B[X]_{S}=B[X]_{N+XB[X]}=(B_{N}[X])_{N+XB[X]}$ is a local PID, which is a DVR. Since $R_{M}$ is a pullback domain $A_{P}\times_{B_{N}}B[X]_{S}$ \cite[Proposition 1.9]{f80}, $R_{M}$ is a Mori PVD by \cite[Proposition 2.6]{ad} and Corollary \ref{accp0}. If $M\in S_{2}$, then $M=Q\cap R$ for some maximal ideal $Q$ of $B[X]$ such that $X\not\in Q$, and $R_{M}\cong B[X]_{Q}$. Since $B$ is von Neumann regular, $B[X]_{Q}$ is a DVR as mentioned in the proof of $(1)\Rightarrow (2)$ of \cite[Theorem 6]{a76}. In particular, $R_{M}$ is a Mori PVD. Therefore $(a)$ follows.\\
\\
$(a)\Rightarrow(d)$: Follows from definition.\\
\\
$(d)\Rightarrow(e)$: Follows from \cite[Lemma 1.(a)]{abd97}.\\
\\
(3): Follows from (1) and Lemma \ref{keylemma}.
\\
\\
(4): $(a)\Rightarrow(c)$: Follows from \cite[Theorem 2]{j66}.\\
\\
$(c)\Rightarrow (d)$: Suppose that $R$ is a one-dimensional arithmetical ring. Then by (1), $A\subseteq B$ is a unibranched extension of von Neumann regular rings. Now, let $P$ be a maximal ideal of $A$ and $M=P+XB[X]$. Then there exists unique maximal ideal $N$ of $B$ such that $N\cap A=P$, and $R_{M}$ is the pullback domain $A_{P}\times_{B_{N}}B[X]_{N+XB[X]}$ as mentioned in the proof of $(f)\Rightarrow (a)$ of (2). Since $R$ is arithmetical, $R_{M}$ is a valuation domain, and we must have $A_{P}=B_{N}$ \cite[Proposition 1.1.8.(1)]{fhp}. Since $B/A$ is an $A$-module, and $(B/A)_{P}=B_{N}/A_{P}=0$ for each prime ideal $P$ of $A$, we have $B/A=0$ and $A=B$.\\
\\
$(d)\Rightarrow (a)$: Follows from \cite[Theorem 18.7]{g84}. \\
\\
$(b)\Rightarrow(c)$: Trivial.\\
\\
$(d)\Rightarrow (b)$: Follows from \cite[Theorem 6]{a76}.\\
\\
(5): $(a)\Rightarrow(b)$: Trivial.\\ 
\\
$(b)\Rightarrow(c)$: Follows from \cite[Proposition 33]{c211}.\\
\\
$(c)\Leftrightarrow (d)$: Follows from \cite[Lemma 20.(2)]{c211}.\\
\\
$(d)\Rightarrow(e)$: Suppose that $R$ is a finite-absorbing locally divided ring. Then $R$ satisfies $(*)$ \cite[Proposition 25]{c211}. Hence by (2),  $A\subseteq B$ is a unibranched extension of von Neumann regular rings. On the other hand, since $R$ is finite-absorbing, so is $A$ \cite[Lemma 18.(4)]{c211}, and $A$ has only finitely many minimal prime ideals \cite[Theorem 2.5]{Anderson}. Since $A$ is zero-dimensional, $A$ is semilocal, and so is $B$ since $A\subseteq B$ is unibranched. 
\\
\\
$(e)\Rightarrow(f)$: Assume $(e)$. Since $A$ and $B$ are semilocal von Neumann regular rings, they are Artinian reduced rings. Hence $A$ (respectively, $B$) must be isomorphic to finite product of fields $F_{1}\times\cdots\times F_{n}$ (respectively, $L_{1},\dots, L_{m}$). Since $A\subseteq B$ is unibranched, we may assume that $n=m$. In fact, for each $i\in\{1,\dots, n\}$, every maximal ideal of $A$ (respectively, $B$) is of the form $M_{i}=\prod\limits_{j=1}^{n}S_{j}$ where $S_{j}=F_{j}$ if $j\neq i$ and $F_{i}=\{0\}$ (respectively, $N_{i}=\prod\limits_{j=1}^{n}T_{j}$ where $T_{j}=L_{j}$ if $j\neq i$ and $T_{i}=\{0\}$). Hence by reindexing $\{F_{1}\dots, F_{n}\}$ if necessary, we may assume that $M_{i}=N_{i}\cap A$ for each $i$. Then
for a fixed $i\in\{1,\dots, n\}$, the ring $F_{i}\cong A/M_{i}$ can be embedded to $L_{i}\cong B/N_{i}$. Hence we may assume that $F_{i}\subseteq L_{i}$.
\\
\\
$(f)\Rightarrow (e)$:  It is well-known that a finite product of fields is semi-quasilocal von Neumann regular rings. Hence if $(f)$ is true, then $A\subseteq B$ is an extension of semi-quasilocal von Neumann regular rings. It is also a unibranched extension since $F_{i}\subseteq L_{i}$ is unibranched for each $i\in\{1,\dots, n\}$.\\
\\
$(e)\Rightarrow(a)$: Suppose that $(e)$ is true. Then $R$ is a strongly Laskerian LPVR by (2) and (3). Hence $R$ is a TAF-ring by Theorem \ref{taf54}.\\ 
\\
(6): $(a)\Rightarrow(b)$: Suppose that $R$ is a general ZPI-ring. Then $R$ has Krull dimension one \cite[p.469]{G}, and $R$ is a strongly Laskerian arithmetical ring \cite[Corollary 36]{c211}, so $A=B$ is a von Neumann regular ring by (4). Therefore $R$ is strongly Laskerian by (5), and must be a B{\'e}zout ring by (4).\\
\\
$(b)\Leftrightarrow(c)$: Assume $(b)$. Since every B{\'e}zout ring is locally divided, $R$ must be a one-dimensional TAF-ring by (2) and (5). Hence by (4) we have $A=B$.\\
\\
$(c)\Leftrightarrow(d)$: Follows from (5).\\
\\
$(c)\Rightarrow (a)$: If $R$ is a TAF-ring and $A=B$, then $R$ is a general ZPI-ring \cite[Corollary 4.4]{adk}. 
\end{proof}

In a univariate polynomial ring, several statements of Proposition \ref{coro56} are equivalent.

\begin{theorem}
\label{theorem54}
The following are equivalent for a ring $R$.
\begin{enumerate}
    \item $R$ is a von Neumann regular ring.
    \item $R[X]$ is a B{\'e}zout ring.
    \item $R[X]$ is an arithmetical ring.
    \item $R[X]$ satisfies $(*)$.
    \item $R[X]_{M}$ is a DVR for each maximal ideal $M$ of $R[X]$.
    \item $R[X]_{M}$ is Mori PVD  for each maximal ideal $M$ of $R[X]$.
    \item $R[X]_{M}$ is a PVD for each maximal ideal $M$ of $R[X]$.
    \item $R[X]$ is a LPVR.
    \item $R[X]_{M}$ is a TAF-domain  for each maximal ideal $M$ of $R[X]$.
    \item $R[X]_{M}$ is a TAF-ring  for each maximal ideal $M$ of $R[X]$.
    \item $R[X]_{M}$ is a FAF-ring  for each maximal ideal $M$ of $R[X]$.
    \item $R[X]_{M}$ is a one-dimensional domain for each maximal ideal $M$ of $R[X]$.
    \item $R[X]$ is a locally divided ring.
    \item $R[X]$ is a one-dimensional reduced ring.
\end{enumerate}
\end{theorem}

\begin{proof}

$(1)\Rightarrow(2)$: \cite[Theorem 18.7]{g84}.\\
\\
$(2)\Rightarrow(3)$: \cite[Theorem 2]{j66}.\\
\\
$(5)\Rightarrow(6)\Rightarrow(7)\Rightarrow(8)$: Well-known. \\
\\
$(8)\Rightarrow(13)$: \cite[Lemma 1.(a)]{abd97}. \\
\\
$(6)\Leftrightarrow(9)$: Proposition \ref{tafmori}.\\
\\
$(9)\Rightarrow(10)\Rightarrow(11)\Rightarrow(12)\Rightarrow(13)$: Follows from definition.\\
\\
$(1)\Leftrightarrow(4)\Leftrightarrow (13)$: \cite[Theorem 60]{c211}. \\
\\
$(1)\Leftrightarrow(3)\Leftrightarrow(5)\Leftrightarrow(14)$: Proposition \ref{coro56}.(1).
\end{proof}

We end this paper with an example showing how we can construct an integral domain whose overrings are FAF-domains using the $A+XB[X]$ construction.

\begin{coro}
\label{ccc45}
Let $A\subseteq B$ be an extension of integral domains, $X$ an indeterminate, and $R=A+XB[X]$. Then the following are equivalent.
\begin{enumerate}
\item $A\subseteq B$ is an algebraic field extension. 
\item Every overring of $R$ is a TAF-domain.
\item Every overring of $R$ is a FAF-domain.
\end{enumerate}
\end{coro}

\begin{proof}
$(1)\Rightarrow (2)$: Suppose that $A\subseteq B$ is an algebraic field extension. Then $R$ is a TAF-domain by Proposition \ref{coro56}.(5).  Moreover, $R'=R^{*}=B[X]$ \cite[Lemma 1.1.4.(9)]{fhp}. Hence every overring of $R$ is a TAF-domain by Proposition \ref{tafmori} and Corollary \ref{tafover}.\\
\\
$(2)\Rightarrow (3)$: Trivial.\\
\\
$(3)\Rightarrow (1)$: Suppose that every overring of $R$ is a FAF-domain. Then $A$ and $B$ are fields by Proposition \ref{coro56}.(5). To prove that $A\subseteq B$ is algebraic, notice that if there is an element $t$ of $B$ that is transcendental over $A$, then $A[t]+XB[X]$ is a two-dimensional overring of $R$ by Lemma \ref{con56}.(3), which contradicts our assumption since every FAF-domain has Krull dimension at most one \cite[Theorem 4.1]{adk}.
\end{proof}

\section*{acknowledgement}
This work was supported by the National
Research Foundation of Korea (NRF) grant funded by the Korea government (MSIT)(No.
2022R1C1C2009021).

\end{document}